\pdfoutput=1
\newif\ifjournal
\newif\ifwias
\newif\ifarxiv
\newif\ifpdf
\journalfalse
\wiasfalse
\arxivtrue
\pdftrue

\ifjournal

	\documentclass[a4paper,10pt,oneside,onecolumn,sort&compress,preprint,3p]{elsarticle}
	\usepackage{amsmath,amsthm,thmtools}
	\usepackage{lineno}
	\usepackage[utf8]{inputenc} 
	\usepackage[T1]{fontenc}
	\usepackage{lmodern}
	\usepackage[linktocpage=true,colorlinks=true,linkcolor=blue,citecolor=red,
				urlcolor=blue,bookmarks=false]{hyperref}

	\RequirePackage[capitalize,nameinlink]{cleveref}[0.19]
\else\ifwias

	\documentclass[a4paper,12pt,twoside,english]{article}
	\usepackage{wiaspreprint}

	\pagefootleft{WIAS Preprint No. 2805}

	\usepackage{amsmath,amsthm,thmtools}
	\usepackage[utf8]{inputenc} 
	\usepackage[T1]{fontenc}
	\usepackage{lmodern}
	\usepackage{cleveref}

	\author[B. Tak\'acs, Y. Hadjimichael]{%
	B\'alint Tak\'acs\footnote{Applied Analysis and Computational Mathematics \\
		E\"otv\"os Lor\'and University \\ P\'azm\'any P\'eter s\'et\'any 1/C \\ 1117 Budapest \\ Hungary \\
	E-Mail: takacsbm@caesar.elte.hu
	},
	Yiannis Hadjimichael\footnote{Weierstrass Institute \\
	Mohrenstra{\ss}e 39 \\ 10117 Berlin \\ Germany \\
	E-Mail: yiannis.hadjimichael@wias-berlin.de	
	}
	}
	\title[High order discretizations for spatial-dependent epidemic models]
	{High order discretization methods for spatial-dependent epidemic models}
	\nopreprint{2805}	
	\nopreyear{2021}	
	\date{October 8, 2021}	
	\thanks{This work was funded by the European Union, co-financed by the European Social Fund under
		contract No.~EFOP-3.6.3-VEKOP-16-2017-00002, and partially supported by the Leibniz
		competition}

\else\ifarxiv

	\documentclass[11pt,a4paper,oneside]{article}

	\usepackage{amsmath,amsfonts,amsthm,thmtools}
	\usepackage[pdftex]{graphicx} 
	\usepackage[hmargin=1.0in,bottom=1.8in,footskip=1.0in,marginparwidth=2cm]{geometry}
	\usepackage[usenames]{color} 
	\usepackage[sc]{mathpazo} 
	\usepackage{wrapfig}
	\usepackage[utf8]{inputenc} 
	\usepackage[T1]{fontenc}
	\usepackage{lmodern}
	\usepackage[colorinlistoftodos,prependcaption,textsize=tiny]{todonotes}
	\usepackage[
		style=numeric-comp,
		hyperref=true,
		backend=biber,
		doi=false,
		url=false,
		isbn=false,
		backref=false,
		giveninits=true,
		sorting=nyt,
		maxcitenames=3,
		maxbibnames=100,
		block=none]{biblatex}
	\usepackage[linktocpage=true,colorlinks=true,linkcolor=blue,citecolor=red,
				urlcolor=blue,bookmarks=false]{hyperref}
	\RequirePackage[capitalize,nameinlink]{cleveref}[0.19]

	\DeclareNameAlias{default}{family-given}
	\newbibmacro*{string+doiurlisbn}[1]{
		\iffieldundef{doi}{
			\iffieldundef{url}{
				\iffieldundef{isbn}{
					\iffieldundef{issn}{#1}
						{\href{http://books.google.com/books?vid=ISSN\thefield{issn}}{#1}}
				}
						{\href{http://books.google.com/books?vid=ISBN\thefield{isbn}}{#1}}
			}
			{\href{\thefield{url}}{#1}}
		}
		{\href{http://dx.doi.org/\thefield{doi}}{#1}}
	}
	\DeclareFieldFormat*{title}{\usebibmacro{string+doiurlisbn}{#1\isdot}}
	\renewcommand*{\newunitpunct}{\addcomma\space}
	\DeclareFieldFormat{journaltitle}{\mkbibemph{#1}\newunitpunct}
	\renewbibmacro{in:}{}
	\AtEveryBibitem{\clearfield{number}}
	\DeclareSourcemap{
		\maps[datatype=bibtex]{
			\map[overwrite=true]{
				\step[fieldsource=fjournal]
				\step[fieldset=journal, origfieldval]
			}
		}
	}
	\title{High order discretization methods for spatial-dependent epidemic models}
	\author{B\'alint Tak\'acs\footnote{\sloppy Applied Analysis and Computational Mathematics,
	MTA-ELTE Numerical Analysis and Large Networks Research Group,
	E\"otv\"os Lor\'and University,
	P\'azm\'any P\'eter s\'et\'any 1/C, Budapest 1117, Hungary, and
	Department of Differential Equations,
	Budapest University of Technology and Economics,
	Egry J\'ozsef u. 1, Budapest 1111, Hungary,
	(email: \texttt{takacsbm@caesar.elte.hu}).}
	\and
	Yiannis Hadjimichael\footnote{MTA-ELTE Numerical Analysis and Large Networks Research Group,
	E\"otv\"os Lor\'and University,
	P\'azm\'any P\'eter s\'et\'any 1/C, Budapest 1117, Hungary,
	(email: \texttt{hadjimy@cs.elte.hu}), and
	Weierstrass Institute (WIAS), Mohrenstra{\ss}e 39, Berlin 10117, Germany,
	(email: \texttt{yiannis.hadjimichael@wias-berlin.de}).}
	}

\fi\fi\fi

\usepackage{amssymb,mathtools}
\usepackage[centerlast]{subfigure}
\usepackage{booktabs}
\usepackage[ntheorem]{empheq}
\usepackage{bm}
\usepackage{bigints}
\usepackage[norefs,nocites]{refcheck}
\usepackage{array}
\usepackage{multirow}

\DeclareFontFamily{OMX}{lmex}{}
\DeclareFontShape{OMX}{lmex}{m}{n}{<-> lmex10}{}

\declaretheorem[style=plain,numberwithin=section]{theorem}
\declaretheorem[style=plain,numberwithin=section]{lemma}
\declaretheorem[style=plain,numberwithin=section]{proposition}
\declaretheorem[style=plain,numberwithin=section]{corollary}

\declaretheorem[style=definition,numberwithin=section]{remark}
\numberwithin{equation}{section}
\numberwithin{table}{section}
\numberwithin{figure}{section}

\newcommand\norm[1]{\left\lVert#1\right\rVert}
\newcommand{\tfinal}{t_\text{f}}
\newcommand{\R}{\mathbb{R}}
\newcommand{\bigO}[1]{\mathcal{O}(#1)}
\newcommand{\sspcoef}{\mathcal{C}}
\newcommand{\balpha}{\bm{\alpha}}
\newcommand{\bv}{\bm{v}}
\newcommand{\ud}{\,\mathrm{d}}
\makeatletter
\newcommand{\Biggg}{\bBigg@{2.8}}
\makeatother
\newcolumntype{M}[1]{>{\centering\arraybackslash}m{#1}}
\newcolumntype{L}[1]{>{\raggedright\arraybackslash}m{#1}}

\makeatletter
\newcommand{\refcheckize}[1]{%
  \expandafter\let\csname @@\string#1\endcsname#1%
  \expandafter\DeclareRobustCommand\csname relax\string#1\endcsname[1]{%
    \csname @@\string#1\endcsname{##1}\@for\@temp:=##1\do{\wrtusdrf{\@temp}\wrtusdrf{{\@temp}}}}%
  \expandafter\let\expandafter#1\csname relax\string#1\endcsname
}
\newcommand{\refcheckizetwo}[1]{%
  \expandafter\let\csname @@\string#1\endcsname#1%
  \expandafter\DeclareRobustCommand\csname relax\string#1\endcsname[2]{%
    \csname @@\string#1\endcsname{##1}{##2}\wrtusdrf{##1}\wrtusdrf{{##1}}\wrtusdrf{##2}\wrtusdrf{{##2}}}%
  \expandafter\let\expandafter#1\csname relax\string#1\endcsname
}
\makeatother

\refcheckize{\cref}
\refcheckize{\Cref}
\refcheckizetwo{\crefrange}
\refcheckizetwo{\Crefrange}

\crefname{section}{section}{sections}
\crefname{subsection}{section}{subsections}
\Crefname{subsection}{Section}{Subsections}

\begin{document}

\ifjournal

	\begin{frontmatter}

		\title{High order discretization methods for spatial-dependent epidemic models\tnoteref{t1}}
		\tnotetext[t1]{This work was funded by the European Union, co-financed by the European Social Fund
			under contract No.~EFOP-3.6.3-VEKOP-16-2017-00002, and partially supported by the Leibniz
			competition.}

		\author[1,2]{B\'alint Tak\'acs\texorpdfstring{\corref{cor1}}{}}
		\ead{takacsbm@caesar.elte.hu}
		\author[1,3]{Yiannis Hadjimichael}
		\ead{yiannis.hadjimichael@wias-berlin.de}

		\cortext[cor1]{Corresponding author}
		\address[1]{MTA-ELTE Numerical Analysis and Large Networks Research Group,
			E\"otv\"os Lor\'and University, \\
			P\'azm\'any P\'eter s\'et\'any 1/C, Budapest 1117, Hungary}
		\address[2]{Department of Differential Equations, Budapest University of Technology and Economics, \\
			Egry J\'ozsef u. 1, Budapest 1111, Hungary}
		\address[3]{Weierstrass Institute for Applied Analysis and Stochastics (WIAS), Mohrenstra{\ss}e 39,
			Berlin 10117, Germany}


\else

	\maketitle

\fi

\begin{abstract}
	In this paper, an epidemic model with spatial dependence is studied and results regarding its stability
	and numerical approximation are presented.
	We consider a generalization of the original Kermack and McKendrick model in which the size of the
	populations differs in space.
	The use of local spatial dependence yields a system of partial-differential equations with integral
	terms.
	The uniqueness and qualitative properties of the continuous model are analyzed.
	Furthermore, different spatial and temporal discretizations are employed, and step-size restrictions for
	the discrete model's positivity, monotonicity preservation, and population conservation are investigated.
	We provide sufficient conditions under which high-order numerical schemes preserve the stability of the
	computational process and provide sufficiently accurate numerical approximations.
	Computational experiments verify the convergence and accuracy of the numerical methods.
\end{abstract}

\ifjournal
	\end{frontmatter}
	\linenumbers
\fi

\section{Introduction}
During the course of human history, many epidemics have ravaged the population.
Since the plague of Athens in 430 BC described by historian Thucydides (one of the earliest descriptions of
such epidemics), researchers tried to model and explain the outbreak of illnesses.
More recently, the outbreak of the COVID-19 pandemic revealed the importance of epidemic research and
the development of models to describe the public health impact of major virus diseases.

Nowadays, many of the models used in science are derived from the original ideas of Kermack and
McKendrick \cite{Kermack/McKendrick:1927:MathematicalEpidemics} in 1927, who constructed a
compartment model to study the process of epidemic propagation.
In their model, usually referred to as the SIR model, the population is split into three classes: $S$ being
the group of healthy individuals who are susceptible to infection; $I$ is the compartment of the ill species
who can infect other individuals; and $R$ being the class of recovered or immune individuals.
The original model of Kermack and McKendrick took into account constant rates of change and neglected any
natural deaths and births or vaccination.
In this work, we also consider constant rates of change, and in addition, we include the term $c\,S(t)$ to
describe immunization effects through vaccination.
The SIR model takes the form
\begin{align}\label{sir}
	\left\{\begin{aligned}
		\frac{d}{dt}S(t) &= -a\,S(t)I(t) - c\,S(t), \\
		\frac{d}{dt}I(t) &= a\,S(t)I(t) - b\,I(t), \\
		\frac{d}{dt}R(t) &= b\,I(t) + c\,S(t),
	\end{aligned}\right.
\end{align}
where the positive constant parameters $a$, $b$ and $c$ correspond to the rate of infection, recovery and
vaccination, respectively.

Since the introduction of the model \eqref{sir} in 1927, numerous extensions were constructed to describe
biological processes more efficiently and realistically.
A natural extension is to take into account the heterogeneity of the domain so that we examine not only the
change of the populations in time but also observe the spatial movements.
Kendall introduced such models that transformed the system of ordinary differential equations \eqref{sir}
into a system of partial differential equations \cite{Bartlett:1957:MeaslesPeriodicity,
Kendall:1965:ModelsInfection}.

The time-dependent functions in \eqref{sir} represent the number of individuals in each class but contain
no information about their spatial distribution.
Instead, one can replace these concentration functions with spatial-dependent functions describing the
density of healthy, infectious, and recovered species over some domain $\Omega \subset \R^d$
\cite{Takacs/Horvath/Farago:2019:SpaceModelsDiseases}.
In this paper, we consider a bounded domain in $\R^2$; hence the system \eqref{sir} is recast as
\begin{align}\label{sir2}
	\left\{\begin{aligned}
		\frac{\partial}{\partial t}S(t,x,y) &= -a\,S(t,x,y)I(t,x,y) - c\,S(t,x,y), \\
		\frac{\partial}{\partial t}I(t,x,y) &= a\,S(t,x,y)I(t,x,y) - b\,I(t,x,y), \\
		\frac{\partial}{\partial t}R(t,x,y) &= b\,I(t,x,y) + c\,S(t,x,y).
	\end{aligned}\right.
\end{align}
However, the model \eqref{sir2} is still insufficient as it does not allow the disease to spread in the
domain but only accounts for a point-wise infection.
Spatial points do not interact with each other but infect species only at their location.
To allow a realistic propagation of the infection, we assume that an infected individual can
spread the disease on susceptible species in a certain area around its location.
Let us define a non-negative function with compact support
\begin{align}\label{Gfunc}
	G(x, y, r, \theta) =
		\begin{cases}
			g_1(r) g_2(\theta,x,y), & \text{if } \bigl(\bar{x}(r,\theta), \bar{y}(r,\theta)\bigr) \in
			B_{\delta}(x,y), \\
			0,  & \text{otherwise,}
		\end{cases}
\end{align}
that describes the effect of a single point $(x,y)$ in a $\delta$-radius neighborhood
$B_{\delta}(x,y)$, and set $\bar{x}(r,\theta) = x + r\cos(\theta)$ and
$\bar{y}(r,\theta) = y + r\sin(\theta)$.
The function $G(x, y, r, \theta)$ demonstrates how healthy individuals at points
$(\bar{x}(r,\theta),\bar{y}(r,\theta))$ are infected by the center point $(x,y)$, where
$r \in [0,\delta]$ is the distance from the center and $\theta \in [0,2\pi)$ is the angle.
In this work we consider $G(x, y, r, \theta)$ to be a separable function.
The effect the center point $(x,y)$ has at a distance $r$ is described by $g_1 (r)$; a decreasing,
non-negative function that is equal to zero for values $r \geq \delta$ (since there is no effect outside
$B_{\delta}(x,y)$).
The function $g_2 (\theta,x,y)$ characterizes the angular effect, i.e., the effect at an angle $\theta$ with
respect to the center point $(x,y)$.
The case of a spatially independent function $g_2(\theta,x,y)$, that is the same for all $(x,y) \in \Omega$,
is widely studied in \cite{Farago/Horvath:2016:SpaceTimeEpidemic} and
\cite{Farago/Horvath:2018:QualitativeDiseasePropagation}.
A more generic function the depends on spatial coordinates could be useful in the case of epidemic diseases
with a given direction of propagation, or in the case of wildfires when the wind profile is known.
Such a function with a constant wind direction was described in
\cite{Takacs/Horvath/Farago:2019:SpaceModelsDiseases}.
In both cases it is supposed that $g_2$ is bounded and periodic in the sense that
$g_2 (0,x,y) = \lim_{\theta \rightarrow 2\pi} g_2 (\theta,x,y)$, for each $(x,y) \in \Omega$.

The nonlinear terms of the right-hand side of \eqref{sir2} describe the interaction of susceptible and
infected species.
We can now utilize \eqref{Gfunc} and replace the density of infected species in these nonlinear terms by
\begin{align*}
	\int_0^{\delta}\int_0^{2\pi}G(x,y,r,\theta)I\bigl(t,\bar{x}(r,\theta),\bar{y}(r,\theta)\bigr) \,r
	\ud\theta \ud r,
\end{align*}
where we used the fact that $G(x,y,r,\theta) = 0$ outside the ball $B_{\delta}(x,y)$.
Therefore, the model \eqref{sir2} can be extended as a system of integro-differential equations
\begin{align}\label{spat}
\left\{
\begin{aligned}
	\frac{\partial S(t,x,y)}{\partial t} &= -S(t,x,y)\int_{0}^{\delta}\int_{0}^{2 \pi}
		g_1(r) g_2(\theta,x,y) I\bigl(t,\bar{x}(r,\theta),\bar{y}(r,\theta)\bigr) \, r \ud\theta \ud r -
		cS(t,x,y),\\
	\frac{\partial I(t,x,y)}{\partial t} &= S(t,x,y) \int_{0}^{\delta} \int_{0}^{2 \pi}
		g_1(r) g_2(\theta,x,y)I\bigl(t,\bar{x}(r,\theta),\bar{y}(r,\theta)\bigr) \, r \ud\theta \ud r -
		bI(t,x,y), \\
	\frac{\partial R(t,x,y)}{\partial t} &= bI(t,x,y) + cS(t,x,y).
\end{aligned}
\right.
\end{align}
We consider homogeneous Dirichlet conditions since we assume that there is no susceptible population outside
of the domain $\Omega$, and there is no diffusion in \eqref{spat}.


\subsection{Outline and scope of the paper}
The aim of this paper is twofold.
First, in \cref{sec:Analytical} we analyze the stability of the continuous model \eqref{spat}
and prove that a unique solution exists under some Lipschitz continuity and boundedness
assumptions.
Secondly, in \cref{sec:Spatial,sec:TimeIntegration} we seek numerical schemes that approximate the
solution of \eqref{spat} and maintain its qualitative properties.

We verify that the analytic solution satisfies biologically reasonable properties; however,
as shown in \cref{subsec:QualitativeModel} the solution can only be expressed implicitly in terms of $S$,
$I$, and $R$, and thus cannot be obtained.
Therefore, the problem must be handled with stable and accurate numerical methods.
A numerical approximation is presented in \cref{subsec:InteqralSolution} that provably satisfies the
solution's properties.
The first order accuracy of this approximation motivates the search for suitable high order
numerical methods that preserve a discrete analog of the properties of the continuous
model.
In \cref{sec:Spatial} we use quadrature formulas to reduce the integro-differential system
\eqref{spat} to an ODE system.
We study the accuracy of different quadratures and interpolation techniques for approximating
the multiple integrals in \eqref{spat}.
Furthermore, the employment of time integration methods yields a system of difference equations.
\Cref{sec:TimeIntegration} shows that a time-step restriction is sufficient and necessary such that
the forward Euler method maintains the stability properties of the ODE system.
We prove that high order strong-stability-preserving (SSP) Runge--Kutta methods can be used
under appropriate restrictions; thus, we can obtain a high order stable scheme both in space
and time.
Finally, in \cref{sec:Numerics} we demonstrate that the numerical experiments confirm the theoretical
conclusions.
The reader can also find a list of symbols and notations used in the paper in the appendix.

\section{Stability of the analytic solution}\label{sec:Analytical}
Analytic results for deterministic epidemic models have been studied by several authors, see for
example, \cite{Kendall:1965:ModelsInfection, Aronson:1977:AsymptoticEpidemic,
Thieme:1977:SpreadEpidemic}.
Such models lie in the larger class of reaction-diffusion problems and therefore one can obtain
theoretical results by studying the more general problem.
In \cite{deMottoni/Orlandi/Tesei:1979:AsymptoticSpatialEpidemics}, de Mottoni et al. considered a
diffusion-reaction epidemic model with a spatial spread of infection and proved the existence of a unique
local solution for arbitrary initial conditions.
Moreover, they proved that if the initial conditions are non-negative, then the solution is global and also
non-negative for all times.
In their paper, the authors assumed a non-vanishing viscosity model and described the spread of the infection
by a non-negative function with compact support in $L^1(\R^2)$ and bounded by unity.
In this section, we prove the uniqueness and global existence of the solution of \eqref{spat} without the
above assumptions and for any initial conditions.
Instead of \cite{deMottoni/Orlandi/Tesei:1979:AsymptoticSpatialEpidemics} we follow the work of Capasso and
Fortunato \cite{Capasso/Fortunato:1980:StabilityEvolutionEqns}, and assume that the nonlinear part of system
\eqref{spat} satisfies certain continuity and boundedness properties.

We consider the following semilinear autonomous evolution problem
\begin{equation}\label{geneq}
\begin{aligned}
	\dfrac{d u}{d t} (t) &= - A u(t) + F(u(t)), \\
	u_0 &= u(0) \in D(A),
\end{aligned}
\end{equation}
where $A$ is a self-adjoint and positive-definite operator in a real Hilbert space $E$ with domain $D(A)$.
Define $\lambda_0 = \inf \sigma (A)$, where $\sigma (A)$ denotes the spectrum of $A$.
Let $\Omega$ be a bounded domain in $\R^2$ and let us choose the space
${E \coloneqq L^2(\Omega) \times L^2(\Omega)}$ with a norm $\norm{ \cdot }$ defined by
\begin{align}\label{norm_E}
	\norm{\left( \begin{aligned} u_1 \\ u_2 \end{aligned} \right)} \coloneqq
	\left(\norm{u_1}_{L^2}^2 + \norm{u_2}_{L^2}^2\right)^{\frac{1}{2}}.
\end{align}
Here $u(t) = \left(u_1(t), u_2(t)\right)^\intercal \in C^1\bigl([0, \tfinal), D(A)\bigr)$, for some final
time $\tfinal$.
We also equip $D(A)$ with the norm
\begin{align}\label{norm_DA}
	\norm{ u(t) }_A \coloneqq \norm{ A u(t)}.
\end{align}

Note that it is sufficient to consider only the first two equations in \eqref{spat}, since $R(t,x,y)$ can
be obtained by using that the sum $S(t,x,y) + I(t,x,y) + R(t,x,y)$ is constant in time for every point
$(x,y)$.
Hence, in view of problem \eqref{spat}, the linear operator $A$ is defined as
\begin{align}\label{operatorA}
	A\left( \begin{aligned} u_1 \\ u_2 \end{aligned} \right) \coloneqq \left(\begin{matrix}
c && 0 \\ 0 && b
\end{matrix}\right)\left( \begin{aligned} u_1 \\ u_2 \end{aligned} \right),
\end{align}
and $D(A) = E$.
Because $b$ and $c$ are positive constants, it is easy to see that $A$ is a self-adjoint and
positive-definite operator.
Similarly, $F: D(A) \rightarrow E$ consists of the nonlinear terms, and is defined as
\begin{equation}\label{Fdef}
	F(u(t)) = F\left( u_1(t), u_2(t) \right) \coloneqq
	\left( \begin{aligned} - u_1(t)\mathcal{F}(u_2 (t)) \\ u_1(t)\mathcal{F}(u_2 (t)) \end{aligned} \right).
\end{equation}
The function $\mathcal{F} : L^2(\Omega) \rightarrow L^2(\Omega)$ contains the integral part of
\eqref{spat} and is given by
\begin{equation}\label{calFdef}
	\mathcal{F}\bigl(u_2(t; x, y)\bigr) \coloneqq \int_{0}^{\delta} \int_{0}^{2 \pi}
		g_1(r) g_2(\theta,x,y)u_2\bigl(t, \bar{x}(r,\theta), \bar{y}(r,\theta)\bigr) \, r \ud\theta \ud r.
\end{equation}
\begin{remark}\label{rem:variable_clarification}
Note that in \eqref{calFdef} the function $u_2(t)$ maps $t \longmapsto u_2(t; x, y) \in L^2(\Omega)$, and
hence $u_2(t; x, y)$ can be understood as a function of $(t, x, y)$, such that
$ \int_\Omega |u_2(t, x, y)|^2 \ud x \ud y < \infty$.
\end{remark}

The main result of this section is Theorem~\ref{thm:uniqueness} stating that a unique strong solution of
system \eqref{spat} exists.
Theorem~\ref{thm:uniqueness} considers the system \eqref{geneq} as a generalization of \eqref{spat} and
its proof relies on the fact that the function $F$ in \eqref{Fdef} is Lipschitz-continuous and bounded in
$\norm{\cdot}_A$.
Therefore, we define the following conditions \cite{Capasso/Fortunato:1980:StabilityEvolutionEqns}:
\begin{itemize}
\item[($A_1$)] $F$ is locally Lipschitz-continuous from $D(A)$ to $D(A)$, i.e.,
$$\norm{F(u) - F(v) }_A \leq \zeta(d) \norm{u - v}_A$$
for all $u,v \in D(A)$ such that $d \geq 0$, and $\norm{u}_A \le d$, $\norm{v}_A \le d$.
\item[($A_2$)] $F$ is bounded, i.e., there exist $\nu \ge 0$ and $\gamma \ge 0$ such that
$$\norm{F(u)}_A \leq \nu \norm{u}_A^{1+\gamma}, \quad \forall u \in D(A).$$
\end{itemize}
We also denote the Lebesgue measure of $\Omega$ by $\mu(\Omega)$, and let
\begin{align*}
	\kappa_1 = \max_{r \in (0, \delta)} \{g_1(r) \}, \quad
	\kappa_2 = \max_{\substack{\theta \in [0, 2 \pi), \\ (x,y) \in \Omega}} \{g_2(\theta,x,y) \},
\end{align*}
and $\psi = \max\{b,c\}/\min\{b^2,c^2\}$.
\begin{theorem}\label{thm:uniqueness}
Consider the system \eqref{spat} and assume that conditions ($A_1$) and ($A_2$) hold.
Then, a unique strong solution of system \eqref{spat} exists on some interval $[0,\tfinal)$.
Moreover, if any initial condition $u_0$ belongs to the set
\begin{align*}
	K = \left\{ u \in E \;\Big\vert\;
		\norm{u}_A < \dfrac{\min\{b,c\}}{\sqrt{2} \, \psi \, \kappa_1 \, \kappa_2 \, \mu(\Omega)} \right\},
\end{align*}
then the zero solution is the unique equilibrium solution of the first two equations in \eqref{spat}.
\end{theorem}
The proof of Theorem~\ref{thm:uniqueness} is a direct consequence of two main results by Capasso and
Fortunato \cite{Capasso/Fortunato:1980:StabilityEvolutionEqns}.
For clarity, we state these two theorems below.
\begin{theorem}\cite[Theorem~1.1]{Capasso/Fortunato:1980:StabilityEvolutionEqns}\label{capth1}
If assumption ($A_1$) holds, then a unique strong solution in $D(A)$ of problem \eqref{geneq} exists in some
interval $[0,\tfinal)$.
\end{theorem}
\begin{theorem}\cite[Theorem~1.3]{Capasso/Fortunato:1980:StabilityEvolutionEqns}\label{capth2}
Let us assume that ($A_1$) and ($A_2$) hold. Then for any $u_0 \in \widetilde{K}$ a global strong solution
in $D(A)$, $u(t)$, of \eqref{geneq} exists.
Moreover the zero solution is asymptotically stable in
\begin{align*}
\widetilde{K} = \left\{
	\begin{array}{ll}
		\left\lbrace u \in D(A) \;\big\vert\; \norm{u}_A < (\lambda_0 / \nu)^{1/\gamma}\right\rbrace, &
			\text{ if }\gamma > 0, \\[5pt]
		D(A), & \text{ if } \gamma = 0 \text{ and } \lambda_0 > \nu.
	\end{array}\right.
\end{align*}
\end{theorem}

In the rest of this section we show that the function $F$, as defined in \eqref{Fdef}, satisfies
conditions ($A_1$) and ($A_2$).
First, to prove that ($A_2$) holds, we make use of some auxiliary lemmas; their proofs appear in
\ifjournal
	\ref{appx:ProofLemmas}.
\else
	appendix \ref{appx:ProofLemmas}.
\fi
\begin{lemma}\label{lem:normequivalence}
Let matrix $A$ defined by \eqref{operatorA}, where $b$ and $c$ are positive constants.
The norms $\norm{\cdot}$ and $\norm{\cdot}_A$ are equivalent, i.e.,
\begin{align*}
	\norm{u} \leq \frac{1}{\min\{b,c\}} \norm{u}_A, \quad \text{and} \quad
	\norm{u}_A \leq \max\{b,c\}\norm{u}.
\end{align*}
\end{lemma}
\begin{lemma}\label{lem:L2boundedness}
Let $\mathcal{F}$ be given by \eqref{calFdef}.
Then, we have that
\ifjournal
	$\norm{\mathcal{F}(u_2)}_{L^2} \leq \nu_{\mathcal{F}}\norm{u_2}_{L^2}$,
\else
	\begin{align*}
		\norm{\mathcal{F}(u_2)}_{L^2} \leq \nu_{\mathcal{F}}\norm{u_2}_{L^2},
	\end{align*}
\fi
where ${\nu_{\mathcal{F}} = \kappa_1 \, \kappa_2 \, \mu(\Omega)}$.
\end{lemma}
\begin{corollary}\label{col:A1}
Consider $F$ given by \eqref{Fdef}.
Then, the condition ($A_1$) holds with
$\zeta(d) = \sqrt{2}/\lambda_0 \, \kappa_1 \, \kappa_2 \, \mu(\Omega) \, d$.
\end{corollary}
\begin{proof}
Due to Lemma \ref{lem:normequivalence}, it is sufficient to prove
\begin{equation}\label{lemma4proof}
\norm{F(u) - F(v)} \leq \tilde{\zeta} \norm{u - v},
\end{equation}
for some constant $\tilde{\zeta}$.
First, notice that
\begin{align*}
	\norm{F(u) - F(v)} = \norm{\left(\begin{aligned} -u_1\mathcal{F}(u_2) + v_1\mathcal{F}(v_2) \\
		u_1\mathcal{F}(u_2)- v_1\mathcal{F}(v_2) \end{aligned} \right)}
	\leq \sqrt{ 2} \norm{u_1\mathcal{F}(u_2)- v_1\mathcal{F}(v_2)}_{L^2}.
\end{align*}
We can further bound the right-hand-side of the above inequality, yielding
\begin{align*}
\begin{aligned}
\norm{u_1\mathcal{F}(u_2)- v_1\mathcal{F}(v_2)}_{L^2}^2 &= \norm{u_1\mathcal{F}(u_2) -
v_1\mathcal{F}(u_2) + v_1\mathcal{F}(u_2) - v_1\mathcal{F}(v_2)}_{L^2}^2 \\
&\leq \norm{u_1\mathcal{F}(u_2) - v_1\mathcal{F}(u_2)}_{L^2}^2 +
\norm{v_1\mathcal{F}(u_2) - v_1\mathcal{F}(v_2)}_{L^2}^2 \\
&\leq \norm{\mathcal{F}(u_2)}_{L^2}^2 \norm{ u_1 - v_1}_{L^2}^2 + \norm{\mathcal{F}(u_2) -
\mathcal{F}(v_2) }_{L^2}^2 \norm{v_1}_{L^2}^2.
\end{aligned}
\end{align*}
By Lemma~\ref{lem:L2boundedness} and the linearity of $\mathcal{F}$, we respectively have
\begin{align*}
	\norm{\mathcal{F}(u_2)}_{L^2}^2 \norm{ u_1 - v_1}_{L^2}^2 \leq
		\nu_\mathcal{F}^2 \norm{u_2}_{L^2}^2 \norm{ u_1 - v_1}_{L^2}^2, \\
	\norm{\mathcal{F}(u_2) - \mathcal{F}(v_2) }_{L^2}^2 \norm{v_1}_{L^2}^2 \leq
		\nu_\mathcal{F}^2 \norm{u_2 - v_2}_{L^2}^2 \norm{v_1}_{L^2}^2.
\end{align*}
Let us use the notation $d \in \mathbb{R}^+$ for such a number $d \ge 0$ for which $\norm{u}_A \leq d$ and
${\norm{v}_A \leq d}$.
Then, by definition of norm \eqref{norm_E} we have that $\norm{v_1}_{L^2} \leq \tilde{d}$ and
$\norm{u_2}_{L^2} \leq \tilde{d}$, where ${\tilde{d} = d/\max\{b,c\}}$.
Putting all together and using $\nu_\mathcal{F} = \kappa_1 \, \kappa_2 \, \mu(\Omega)$, we
get
\begin{align*}
	\norm{F(u) - F(v)} &\leq \sqrt{ 2} \norm{u_1\mathcal{F}(u_2)- v_1\mathcal{F}(v_2)}_{L^2} \\
	&\leq \sqrt{2} \, \tilde{d} \, \nu_{\mathcal{F}} \left(\norm{u_1 - v_1}_{L^2}^2 +
		\norm{u_2 - v_2}_{L^2}^2\right)^{1/2} \\
	&\leq \sqrt{2} \, \tilde{d} \, \kappa_1 \, \kappa_2 \, \mu(\Omega) \norm{u-v}.
\end{align*}
Therefore, the inequality \eqref{lemma4proof} holds with
$\tilde{\zeta} = \left(\sqrt{2}/\max\{b,c\}\right) \kappa_1 \, \kappa_2 \, \mu(\Omega) \, d$,
and condition ($A_1$) is satisfied with Lipschitz constant
$\zeta(d) = \sqrt{2}/\lambda_0 \, \kappa_1 \, \kappa_2 \, \mu(\Omega) \, d$, where
$\lambda_0 = \inf \sigma (A) = \min\{b,c\}$.
\end{proof}
\begin{corollary}\label{col:A2}
	Consider $F$ given by \eqref{Fdef}.
	\ifjournal
		Then, the condition ($A_2$) holds with $\nu= \sqrt{2} \, \psi \, \kappa_1 \, \kappa_2 \, \mu(\Omega)$
		and $\gamma=1$.
	\else
		Then, the condition ($A_2$) holds with $\gamma=1$ and
		$\nu= \sqrt{2} \, \psi \, \kappa_1 \, \kappa_2 \, \mu(\Omega)$.
	\fi
\end{corollary}
\begin{proof}
Because of Lemma \ref{lem:normequivalence}, it is enough to prove
\begin{equation}\label{lemma3proof}
\norm{F(u)} \leq \tilde{\nu} \norm{u}^2,
\end{equation}
for some constant $\tilde{\nu}$.
We first have that
\begin{align*}
\norm{F(u)} = \norm{ \left( \begin{aligned} -u_1\mathcal{F}(u_2) \\ u_1\mathcal{F}(u_2) \end{aligned}
\right) } &= (\norm{u_1\mathcal{F}(u_2)}_{L^2}^2 + \norm{u_1\mathcal{F}(u_2)}_{L^2}^2)^{1/2} \\
&\leq \sqrt{2} \norm{u_1}_{L^2} \norm{\mathcal{F}(u_2)}_{L^2}.
\end{align*}
Observe that Lemma~\ref{lem:L2boundedness} can be used to bound $\norm{\mathcal{F}(u_2)}_{L^2}$ from
above, yielding
\begin{align*}
	\norm{F(u)} \leq \sqrt{2} \, \nu_{\mathcal{F}} \norm{u_1}_{L^2}\norm{u_2}_{L^2},
\end{align*}
where $\nu_{\mathcal{F}}$ is defined in Lemma~\ref{lem:L2boundedness}.
Finally, we have that
\ifjournal
	$\norm{u_1}_{L^2}\norm{u_2}_{L^2} \leq \norm{u_1}_{L^2}^2 + \norm{u_2}_{L^2}^2 = \norm{u}^2$,
\else
	\begin{align*}
		\norm{u_1}_{L^2}\norm{u_2}_{L^2} \leq \norm{u_1}_{L^2}^2 + \norm{u_2}_{L^2}^2 = \norm{u}^2,
	\end{align*}
\fi
and thus inequality \eqref{lemma3proof} holds with
$\tilde{\nu}=\sqrt{2} \, \nu_{\mathcal{F}} = \sqrt{2} \, \kappa_1 \, \kappa_2 \, \mu(\Omega)$.
The result follows by using the equivalence of norms from Lemma~\ref{lem:normequivalence}.
\end{proof}

Corollaries~\ref{col:A1} and \ref{col:A2} show that function \eqref{Fdef} satisfies conditions ($A_1$) and
($A_2$).
We also know from Corollary~\ref{col:A2} that $\gamma=1$, so the set $\widetilde{K}$ in Theorem~\ref{capth2}
can be computed by using that $D(A) = E$ and
\begin{align*}
	\left(\dfrac{\lambda_0}{\nu}\right)^{1/\gamma} =
		\dfrac{\min\{b,c\}}{\sqrt{2} \, \psi \, \kappa_1 \, \kappa_2 \, \mu(\Omega)},
\end{align*}
where $b$ and $c$ are the diagonal elements of matrix $A$ in \eqref{operatorA},
$\lambda_0 = \inf \sigma (A)$, and $\psi$, $\kappa_1$, $\kappa_2$ are as defined before.
Finally, it is evident that Theorem~\ref{thm:uniqueness} follows from Theorems~\ref{capth1}
and~\ref{capth2}.

\subsection{Qualitative behavior of the model}\label{subsec:QualitativeModel}
When deriving a mathematical model to describe the spread of an epidemic in both space and time, it is
essential that the real-life processes are being represented as accurately as possible.
More precisely, numerical discretizations applied to such models should preserve the qualitative
properties of the original epidemic model.

The first and perhaps the most natural property is that the number of each species is non-negative at every
time and point of the domain.
Next, assuming that the births and natural deaths are the same, the total number of species of all
classes should be conserved.
Finally, the last properties describe the monotonicity of susceptible and recovered species.
Since an individual moves to the recovered class after the infection, the number of susceptibles cannot
increase in time.
Similarly, the number of recovered species cannot decrease in time.
The above properties are expressed as follows:
\begin{enumerate}
	\item[$C_1$:] The densities $X(t,x,y)$, $X \in \{S, I, R\}$, are non-negative at every point
		$(x,y) \in \Omega$.
	\item[$C_2$:] The sum $S(t,x,y) + I(t,x,y) + R(t,x,y)$ is constant in time for all points
		$(x,y) \in \Omega$.
	\item[$C_3$:] Function $S(t,x,y)$ is non-increasing in time at every $(x,y) \in \Omega$.
	\item[$C_4$:] Function $R(t,x,y)$ is non-decreasing in time at every $(x,y) \in \Omega$.
\end{enumerate}

Before we discuss the preservation of properties $C_1$--$C_4$, we consider the following auxiliary system
\begin{subequations}\label{spat_epsilon}
	\begin{empheq}[left ={\empheqlbrace}]{align}
	\frac{\partial S_{\varepsilon}(t, x, y)}{\partial t} &=
		- S_{\varepsilon}(t, x, y) \mathcal{F}\bigl(I_{\varepsilon}(t;x,y)\bigr)
		- c S_{\varepsilon}(t,x,y), \label{spat_epsilona} \\
	\frac{\partial I_{\varepsilon}(t, x, y)}{\partial t} &=
		S_{\varepsilon}(t, x, y) \mathcal{F}\bigl(I_{\varepsilon}(t;x,y)\bigr)
		- bI_{\varepsilon}(t, x, y) + \varepsilon, \label{spat_epsilonb} \\
	\frac{\partial R_{\varepsilon}(t, x, y)}{\partial t} &= bI_{\varepsilon}(t, x, y)
		+ c S_{\varepsilon}(t,x,y),
	\notag
	\end{empheq}
\end{subequations}
where $0 < \varepsilon \ll 1$.
Because of Remark~\ref{rem:variable_clarification}, $I_{\varepsilon}(t;x,y)$ is equivalent to
$I_{\varepsilon}(t,x,y)$ and there is no ambiguity in \eqref{spat_epsilon}.
The next theorem shows that the solution of \eqref{spat_epsilon} satisfies properties $C_1$--$C_4$.
\begin{theorem}\label{contpres}
	Suppose that the initial conditions of the system \eqref{spat_epsilon} are non-negative, i.e.
	$X_\varepsilon(0,x,y) \geq 0$, $\forall (x,y) \in \Omega$, $X \in \{ S, I, R\}$.
	In such case, the properties $C_1$--$C_4$ hold for the solution of \eqref{spat_epsilon} without any
	restrictions on the time interval $t \in [0, \tfinal]$.
\end{theorem}
\begin{proof}


Let us suppose that the initial conditions assigned to \eqref{spat_epsilon} are all non-negative.
First, we would like to prove the non-negativity of $I_\varepsilon(t,x,y)$ by contradiction.
Assume that the function takes negative values for some time $t$ at some point $(x,y) \in \Omega$.
Define by $t_0$ the last moment in time for which $I_\varepsilon(t,x,y)$ takes non-negative values, i.e.,
$$t_0 \coloneqq \inf \lbrace t \;|\; \exists (x,y) \in \Omega : I_\varepsilon(t,x,y) < 0 \rbrace.$$
By our assumptions, this $t_0$ exists because $I_\varepsilon$ is continuous and the initial conditions are
not negative, i.e., \sloppy{$I_\varepsilon(0,x,y) \ge 0$.}
Because of the continuity of $I_\varepsilon$ and the definition of $t_0$, there is a point $(x_0,y_0)$ for
which $I_\varepsilon(t_0,x_0,y_0) = 0$, and
\begin{equation}\label{negI}
\frac{\partial I_\varepsilon(t_0, x_0, y_0)}{\partial t} \leq 0.
\end{equation}
We know that all the values of $I_\varepsilon$ at $t_0$ inside $B_{\delta}(x_0,y_0)$ are
non-negative by the definition of $t_0$, and $\mathcal{F}\bigl(I_{\varepsilon}(t_0;x_0,y_0)\bigr) \geq 0$
also holds.
Observe that if we consider equation \eqref{spat_epsilonb} at point $(t_0,x_0,y_0)$, then the term
$-b\,I_\varepsilon(t_0,x_0,y_0)$ is zero.
If $\mathcal{F}\bigl(I_{\varepsilon}(t_0;x_0,y_0)\bigr)=0$, then
$\frac{\partial I_\varepsilon(t_0, x_0, y_0)}{\partial t} = \varepsilon > 0$, which is a contradiction; hence
$\mathcal{F}\bigl(I_{\varepsilon}(t_0;x_0,y_0)\bigr) > 0$.
A necessary condition for \eqref{negI} to hold is
$S_\varepsilon(t_0, x_0, y_0) \mathcal{F}\bigl(I_{\varepsilon}(t_0;x_0,y_0)\bigr) \le -\varepsilon$;
therefore, it must be that $S_\varepsilon(t_0, x_0, y_0) < 0$.
Now, dividing \eqref{spat_epsilona} by $S_\varepsilon$ and integrating it with respect to time
$t$ from $0$ to $t_0$, yields
\begin{align*}
	\log{(S_{\varepsilon}(t_0,x,y))} - & \log{(S_{\varepsilon}(0,x,y))} =
	- \int_{0}^{t_0} \mathcal{F}\bigl(I_{\varepsilon}(t_0;x,y)\bigr) \ud t - c t_0.
\end{align*}
By reformulating, and evaluating at point $(x_0,y_0)$ we get that
\begin{align}\label{Skeplet}
	S_{\varepsilon}(t_0,x_0 & ,y_0) = S_{\varepsilon}(0,x_0,y_0)
	\exp{\left(-\int_{0}^{t_0} \mathcal{F}\bigl(I_{\varepsilon}(t_0;x,y)\bigr) \ud t - c t_0\right)}.
\end{align}
Therefore, $S_{\varepsilon}(t_0, x_0, y_0)$ cannot be non-negative so we arrive to a contradiction.

As a result, $I_{\varepsilon}(t,x,y) \geq 0$ for every $t \in [0, \tfinal]$ and $(x,y) \in \Omega$.
Consequently, since $R_{\varepsilon}(0,x,y)$ is non-negative, we have that $R_{\varepsilon}(t,x,y)$ is a
non-decreasing and a non-negative function.
Note also that the derivations leading in formula \eqref{Skeplet} are also true for any time $t$
and point $(x,y) \in \Omega$, meaning that $S_{\varepsilon}$ is also non-negative.
Since $\mathcal{F}\bigl(I_{\varepsilon}(t_0;x,y)\bigr)$ is non-negative, we also get from
\eqref{spat_epsilona} that $S_{\varepsilon}(t,x,y)$ is non-increasing.
\end{proof}

The following Lemma links the systems \eqref{spat} and \eqref{spat_epsilon} and its proof can be found in
\ifjournal
	\ref{appx:ProofLemmas}.
\else
	appendix \ref{appx:ProofLemmas}.
\fi
\begin{lemma}\label{contdep}
	Consider a set of systems \eqref{spat_epsilon} with parameters $\varepsilon_i$, where
	$i \in \{1,2, \dots\}$.
	Assume that the sequence $\left\{ \varepsilon_i \right\}$ tends to zero.
	Then, the solutions $S_{\varepsilon_i}$, $I_{\varepsilon_i}$ and $R_{\varepsilon_i}$ converge in norm to
	the same limit regardless of the choice of the sequence $\left\{ \varepsilon_i \right\}$.
\end{lemma}
\begin{proposition}
	Suppose that the initial conditions of the system \eqref{spat} are non-negative.
	Then, the properties $C_1$--$C_4$ hold for the solution of \eqref{spat} without any
	restrictions on the time interval $t \in [0, \tfinal]$.
\end{proposition}
\begin{proof}
	From Lemma~\ref{contdep} we have that
	\begin{align*}
		\lim_{\varepsilon \rightarrow 0} X_{\varepsilon}(t,x,y)
		\left|_{t \in [0,\tfinal]} - X(t,x,y) \right|_{t \in [0,\tfinal]} = 0
	\end{align*}
	holds for every $X \in \{S, I, R\}$.
	Therefore, the solution of \eqref{spat} depends continuously on the right-hand side of the system of
	equations and hence properties $C_1$--$C_4$ are also satisfied for the system \eqref{spat}.
\end{proof}
\begin{remark}
	The non-negativity of solutions $S$ and $I$ in \eqref{spat} implies exponential convergence to
	zero by taking the sum of the first two equations.
	This confirms the second part of Theorem~\ref{thm:uniqueness}.
\end{remark}

Due to the complicated form of the equations in \eqref{spat}, one can suspect that no analytic solution
can be derived for this system.
Because of this, we are going to use numerical methods to approximate the solution of these equations.
However, the analytic solution of the original SIR model \eqref{sir} has been described in the papers by
Harko et al. \cite{Harko/Lobo/Mak:2014:SIR} and Miller \cite{Miller:2012:EpidemicSizes, Miller:2017:SIR}.
Thus, we can get similar results applying their observations to our modified model \eqref{spat}.
The analytic solution of system \eqref{spat} can be written as
\begin{equation}\label{anasol}
\left\{
\begin{aligned}
 S(t, x, y) &=  S(0, x, y)e^{- \phi(t,x,y) - ct}, \\
 I(t, x, y) &= M_0(x,y) - S(t, x, y) - R(t,x,y), \\
 R(t, x, y) &= R(0,x,y) +  b \int_0^t I (s,x,y) \ud s + c \int_0^t S (s,x,y) \ud s, 
\end{aligned}
\right.
\end{equation} 
where we use the notations 
\begin{align*}
M_0(x,y) &\coloneqq S(0,x,y) + I(0,x,y) + R(0,x,y), \\
\phi(t,x,y) &\coloneqq \int_{0}^{t} \mathcal{F}\bigl(I(s;x,y)\bigr) \ud s,
\end{align*}
and $\mathcal{F}$ is given by \eqref{calFdef}.

It is evident that in \eqref{anasol}, the values of the functions at a given time $t^*$ can only be
computed if the values in the interval $[0, t^*)$ are known.
Consequently, these formulas are not useful in practice, since \eqref{anasol} is an implicit system in the
solutions $S(t, x, y)$, $I(t, x, y)$ and $R(t, x, y)$.
Later (see Table~\ref{first_order_error} in \cref{subsec:Convergence}), an approximation of the
solution of \eqref{anasol} will be compared to the numerical solution of first-order forward Euler scheme.

Since the values of the functions in \eqref{anasol} cannot be calculated directly, numerical methods are
needed to approximate them. We can take two possible paths:
\begin{enumerate}
\item approximate the values of $\phi(t,x,y)$ and the integrals in the third equation of \eqref{anasol} by
	numerical integration; or
\item approximate the solution of the original equation \eqref{spat} by a numerical method.
\end{enumerate} 
The first approach is discussed in \cref{subsec:InteqralSolution}, while the rest of the paper considers
the second case.
We focus on the order and convergence rate of our numerical methods and ensure that qualitative
properties $C_1$--$C_4$ of the analytic solution are preserved by the numerical method.
For that, a discrete analogue of conditions $C_1$--$C_4$ is required; see \cref{sec:TimeIntegration}.

\subsection{Numerical approximation of the integral solution}\label{subsec:InteqralSolution}
As noted before, if we would like to use the solution \eqref{anasol} then we have to approximate the
involved integrals.
This can be achieved by partitioning the time interval $[0,\tfinal]$ into uniform spaced sections by using
a constant time step $\tau$.
With this approach, the integrals can be approximated by a left (right) Riemann sum, and thus consider the
values of densities $X(t,x,y)$, $X \in \{S, I, R\}$, at the left endpoint of each section.
Therefore, for any integer $1 \le n \le \mathcal{N}$ such that $\tfinal = \tau\mathcal{N}$, the integral
of $X(t,x,y)$ can be approximated by
\begin{align*}
	\int_0^{n\tau} X(s,x,y) \ud s \approx
	\left\{
		\begin{aligned}
			&\tau \sum_{k=0}^{n-1} X(k\tau,x,y) \quad (\text{left Riemann sum}), \\
			&\tau \sum_{k=1}^{n} X(k\tau,x,y) \quad (\text{right Riemann sum}).
		\end{aligned}\right.
\end{align*}
An important observation is that the integral equations \eqref{anasol} can be rewritten in a recursive
form
\begin{equation}\label{recurs}
	\left\{
	\begin{aligned}
		S\bigl(n\tau, x, y\bigr) &=  S((n-1)\tau, x, y)
			\exp\left(-\int_{(n-1)\tau}^{n\tau} \mathcal{F}\bigl(I(s;x,y)\bigr) \ud s - c\tau \right), \\
		R\bigl(n\tau, x, y\bigr) &= R((n-1)\tau,x,y) +  b \int_{(n-1)\tau}^{n\tau} I (s,x,y) \ud s +
			c \int_{(n-1)\tau}^{n\tau} S (s,x,y) \ud s,  \\
		I\left(n\tau, x, y\right) &= M_0(x,y) - S(n\tau, x, y) - R(n\tau,x,y).
	\end{aligned}
	\right.
\end{equation} 
Let $X^n(x, y) \approx X\left(n\tau, x, y\right)$, $X \in \{S, I, R\}$, and define
$\mathcal{F}^n \coloneqq \mathcal{F}(I^n)$.
Using the approximations
\ifjournal
	\begin{align*}
		\int_{(n-1)\tau}^{n\tau} \mathcal{F}\bigl(I(s;x,y)\bigr) \ud s \approx \tau\mathcal{F}^{n-1}, \qquad
		\int_{(n-1)\tau}^{n\tau} I (s,x,y) \ud s \approx \tau I^{n-1}, \qquad
		\int_{(n-1)\tau}^{n\tau} S (s,x,y) \ud s \approx \tau S^n,
	\end{align*}
\else
	\begin{align*}
		\int_{(n-1)\tau}^{n\tau} \mathcal{F}\bigl(I(s;x,y)\bigr) \ud s \approx \tau\mathcal{F}^{n-1},
		\int_{(n-1)\tau}^{n\tau} I (s,x,y) \ud s \approx \tau I^{n-1},
		\int_{(n-1)\tau}^{n\tau} S (s,x,y) \ud s \approx \tau S^n,
	\end{align*}
\fi
(note that the first two are left Riemann sum, while the third a right Riemann sum) we get an approximating
scheme for \eqref{anasol}, given by
\begin{subequations}\label{inteq_num}
	\begin{empheq}[left ={\empheqlbrace}]{align}
		S^n &= S^{n-1} e^{-\tau\mathcal{F}^{n-1} - c\tau}, \label{inteq_numa} \\
		R^n &= R^{n-1} + b \tau I^{n-1} + c  \tau S^n, \label{inteq_numb} \\
		I^n &= (S^{n-1} + I^{n-1} + R^{n-1}) - S^n - R^n. \label{inteq_numc}
	\end{empheq}
\end{subequations}
Note that in this case, the order of the equations in \eqref{inteq_num} is important as estimates at time
$t_n = n\tau$ are used to update the rest of solution's components.
\begin{theorem}
Consider the solution $X^n(x, y)$, $X \in \{S, I, R\}$ of scheme \eqref{inteq_num} on the time
interval $[0,\tfinal]$, where $1 \le n \le \mathcal{N}$.
Let $\mathcal{N}$ be the total number of steps such that $\tfinal = \tau\mathcal{N}$, where $\tau$ denotes
the time step.
If the step-size restriction $0 < \tau \le 1/b$ holds, then the solution of \eqref{inteq_num} satisfies
properties $C_1$--$C_4$ at times $t_n = n\tau$, $1 \le n \le \mathcal{N}$.
\end{theorem}
\begin{proof}
We prove the theorem by induction.
Consider the system \eqref{inteq_num} at an arbitrary step $n$ and assume that the properties
$C_1$--$C_4$ hold for the first $n-1$ steps.
First, it is easy to see that the conservation property $C_2$ is satisfied by \eqref{inteq_numc}.
Moreover, by assumption $S^{n-1}$, $I^{n-1}$, and $R^{n-1}$ are non-negative and hence by definition
$\mathcal{F}^{n-1}$ is also non-negative.
As a result, $e^{-\tau(\mathcal{F}^{n-1} + c)} < 1 $, and therefore $S^n$ is non-negative and
monotonically decreasing.
Similarly, the right-hand side terms of \eqref{inteq_numb} are also non-negative, thus $R^n$ is
non-negative and monotonically increasing.
To show that $I^n$ is non-negative, we substitute \eqref{inteq_numa} and \eqref{inteq_numb} into
\eqref{inteq_numc} to get
\begin{align*}
	I^n = S^{n-1}\left(1 - (1+c\tau)e^{-\tau(\mathcal{F}^{n-1}+c)}\right) + I^{n-1}\left(1-b\tau\right).
\end{align*}
We have by assumption that $S^{n-1}$ and $I^{n-1}$ are non-negative; therefore if
\begin{align*}
	1 -  (1 + c\tau)e^{-\tau(\mathcal{F}^{n-1} + c)} \ge 0 \quad \text{and} \quad 1 - b\tau \ge 0,
\end{align*}
then $I^n$ is non-negative.
The function $(1 + c\tau)e^{-\tau(\mathcal{F}^{n-1} + c)}$ is monotonically decreasing for $\tau \ge 0$ (its
derivative is negative) and lies in $(0,1]$; thus $1 -  (1 + c\tau)e^{-\tau(\mathcal{F}^n + c)} \ge 0$ for
any $\tau > 0$.
As a result, the sufficient condition for $I^n$ to remain non-negative is $0 < \tau \le 1/b$.
Note that by using the same arguments as above we can show that conditions $C_1$--$C_4$ hold at the first
step, i.e., $n = 1$, assuming that the initial conditions are non-negative.
This completes the proof.
\end{proof}
%
\vspace*{0pt}
\begin{remark}
	Using left Riemann sums to approximate the integrals in \eqref{recurs} results in local errors of
	order $\mathcal{O}(\tau^2)$.
	Therefore, the solution of \eqref{recurs} can only be first order accurate.
	Notice that it is not possible to use any high order Newton-Cotes formulas since the values of
	$X\left(t, x, y\right)$, $X \in \{S, I, R\}$, are only known at discrete times $t = n\tau$.
\end{remark}
In the next two sections, we discretize \eqref{spat} by first using a numerical approximation of the
integral on the right-hand side of the system, and then applying a time integration method.
This approach results in numerical schemes that are high order accurate, both in space and time.

\section{Spatial discretization}\label{sec:Spatial}
It is evident that the key element of the numerical solution of problem \eqref{spat} is the approximation of
$\mathcal{F}\bigl(I(t,x,y)\bigr)$.
This can be done in two different ways.
The first approach is to approximate the function $I(t, \bar{x}(r,\theta), \bar{y}(r,\theta))$ by a Taylor
expansion, and then proceed further.
This method is studied in \cite{Farago/Horvath:2016:SpaceTimeEpidemic} and
\cite{Farago/Horvath:2018:QualitativeDiseasePropagation}, but is not efficient in the case of non-constant
function $g_2(\theta,x,y)$ as shown in \cite{Takacs/Horvath/Farago:2019:SpaceModelsDiseases}.
The other approach is to use a combination of interpolation and numerical integration (by using quadrature
formulas) to obtain an approximation of $\mathcal{F}\bigl(I(t,x,y)\bigr)$.

We consider two-dimensional quadrature formulas on the disc of radius $\delta$ with positive coefficients.
Denote by $\mathcal{Q}_\delta(x,y)$ the set of quadrature nodes in the disk $B_{\delta}(x,y)$
parametrized by polar coordinates (see \cite{Takacs/Horvath/Farago:2019:SpaceModelsDiseases}), i.e.,
\begin{align*}
	\mathcal{Q}_\delta(x,y) \coloneqq \left\lbrace (x_{ij},y_{ij}) =
	\bigl(x + r_i\cos(\theta_j), y + r_i\sin(\theta_j)\bigr) \in
	B_{\delta}(x,y), i \in \mathcal{I}, j \in \mathcal{J}\right\rbrace,
\end{align*}
where $r_i$ denotes the distance from center point $(x,y)$, $\theta_j$ is the angle, and $\mathcal{I}$ and
$\mathcal{J}$ are the set of indices of quadrature nodes.
Using numerical integration, we get the system
\begin{align}\label{spat_int}
\left \{
\begin{aligned}
\frac{\partial S(t, x, y)}{\partial t} &= - S(t, x, y)T\bigl(t,\mathcal{Q}_\delta(x,y)\bigr) - c S(t,x,y), \\
\frac{\partial I(t, x, y)}{\partial t} &=S(t, x, y) T\bigl(t,\mathcal{Q}_\delta(x,y)\bigr)  - bI(t, x, y), \\
\frac{\partial R(t, x, y)}{\partial t} &= bI(t, x, y) + c S(t,x,y),
\end{aligned}
\right.
\end{align}
where
\begin{align}\label{T_def}
	T\bigl(t,\mathcal{Q}_\delta(x,y)\bigr) = \sum_{(x_{ij},y_{ij}) \in \mathcal{Q}_\delta(x,y)} w_{i,j}
	g_1(r_i) g_2(\theta_j, x, y) I\bigl(t, x + r_i \cos(\theta_j), y + r_i \sin(\theta_j)\bigr),
\end{align}
and $w_{i,j} > 0$ are the weights of the quadrature formula.

\begin{remark}
	Similar arguments as those used in the proof of  Theorem~\ref{contpres} can be applied to system
	\eqref{spat_int}; hence, the properties $C_1$--$C_4$ hold without any restrictions for the analytic
	solution of this system.
	Moreover, it can be easily shown that $T(t,\mathcal{Q}_\delta(x,y))$ satisfies properties ($A_1$) and
	($A_2$), by following the proof of Lemma~\ref{lem:L2boundedness} and linearity arguments.
	As a result system \eqref{spat_int} admits a unique strong solution.
\end{remark}

\subsection{The semi-discretized system}
Now we would like to solve \eqref{spat_int} numerically. The first step is to discretize the problem in
space.
Let us suppose that we would like to solve our problem on a rectangle-shaped domain, namely
$\Omega \coloneqq [0,\mathcal{L}_1] \times [0,\mathcal{L}_2]$.
For our numerical solutions we will discretize this domain by using a spatial grid
\begin{align*}
	\mathcal{G} \coloneqq \left\lbrace(x_k,y_l) \in \Omega\,
		|\; 1\leq k \leq P_1 , 1 \leq l \leq P_2 \right\rbrace,
\end{align*}
which consists of $P_1 \times P_2$ points with spatial step sizes $h_1$ and $h_2$, and approximate the
continuous solutions by a vector of the values at the grid points.
After this semi-discretization, we get the following set of equations
\begin{align}\label{trapeq2}
	\left\{
	\begin{aligned}
	\frac{dS_{k,l}(t)}{dt} &= -S_{k,l}(t)T_{k,l}\bigl(t,\mathcal{Q}_\delta(x_k,y_l)\bigr) - c S_{k,l} (t), \\
	\frac{dI_{k,l}(t)}{dt} &= S_{k,l}(t)T_{k,l}\bigl(t,\mathcal{Q}_\delta(x_k,y_l)\bigr) - bI_{k,l}(t), \\
	\frac{dR_{k,l}(t)}{dt} &= bI_{k,l}(t) + c S_{k,l} (t),
	\end{aligned}
	\right.
\end{align}
where $X_{k,l}(t)$, $X \in \lbrace S,I,R \rbrace$, denotes the approximation of the function at grid point
$(x_k,y_l)$.
The approximation of $\mathcal{F}\bigl(I(t; x_k,y_l)\bigr)$ is denoted by
$T_{k,l}(t,\mathcal{Q}_\delta(x_k,y_l))$ and defined as
\begin{align}\label{quadrature}
\begin{aligned}
	T_{k,l}\bigl(t,&\mathcal{Q}_\delta(x_k,y_l)\bigr) \coloneqq \sum_{(\bar{x}_k,\bar{y}_l) \in
	\mathcal{Q}_\delta(x_k,y_l)} w_{i,j} g_1(r_i)g_2(\theta_j, x_k, y_l)\tilde{I}(t, \bar{x}_k, \bar{y}_l),
\end{aligned}
\end{align}
where $\bar{x}_k = x_k + r_i \cos(\theta_j)$ and $\bar{y}_l = y_l + r_i \sin(\theta_j)$.
Note that the points $(\bar{x}_k, \bar{y}_l)$ might not be included in $\mathcal{G}$; in such case there
are no $I_{k,l}$ values assigned to them.
Because of this, we approximate $I(t, \bar{x}_k, \bar{y}_l)$ by using positivity preserving interpolation
(e.g. bilinear interpolation) with the nearest known $I_{k,l}$ values and positive coefficients.
This is the reason why $\tilde{I}$ is used in \eqref{quadrature} instead of $I$.
\begin{proposition}
A unique strong solution for system \eqref{trapeq2} exists, for which properties $C_1$--$C_4$ hold locally
at a given point $(x_k,y_l)$.
\end{proposition}
\begin{proof}
The proof of existence and uniqueness comes from the Lipschitz continuity and boundness of the right-hand
side, which can be proved similarly as in Corollaries~\ref{col:A2} and \ref{col:A1}.
Properties $C_1$--$C_4$ can be proved in a similar manner as in Theorem \ref{contpres}.
\end{proof}
The next theorem characterizes the accuracy of interpolation and quadrature techniques of system
\eqref{trapeq2}.
\begin{theorem}
Suppose that a quadrature rule approximates the integral \eqref{calFdef} to order $p$, i.e.,
\begin{align}\label{quadr_norm}
\norm{\mathcal{F}\bigl(I(t;x,y)\bigr) - T\bigl(t,\mathcal{Q}(x,y)\bigr)}_{L^2} = \bigO{\delta^p},
\end{align}
where $\delta$ is the radius of the disk in which the integration takes place. Let us suppose that the
(positivity preserving) spatial interpolation $\tilde{I}$ approximates the values of $I$ to order $q$,
i.e.,
\begin{align}\label{interp_norm}
\norm{I(t,x,y) - \tilde{I}(t,x,y)}_{L^2} = \bigO{h^q},
\end{align}
where $h = \min\{h_1,h_2\}$ is the minimum of the spatial step sizes.
Then if $\tilde{u}$ is the solution of \eqref{spat} evaluated at the grid points of $\mathcal{G}$ and
$\tilde{v}$ is the solution of \eqref{trapeq2}, it follows that
\begin{align*}
\norm{\tilde{u} - \tilde{v}}_{l^2} = \bigO{\delta^p} + \bigO{h^q},
\end{align*}
where $l^2$ denotes the discrete $L^2$ norm taken with respect to the spatial variables.
\end{theorem}
\begin{proof}
It is sufficient to prove that if $\mathrm{w}$ is the solution of \eqref{spat_int} evaluated at
the grid points of $\mathcal{G}$, then
\begin{align*}
	\norm{\tilde{u}-\mathrm{w}}_{l^2} = \bigO{\delta^p} \qquad \text{and} \qquad
	\norm{\mathrm{w}-\tilde{v}}_{l^2} = \bigO{h^q}
\end{align*}
hold.
Let us rewrite the first two equations of \eqref{spat_int} in the form
\begin{equation}\label{eq:spatialorderproof1}
\dfrac{ d \mathrm{w}(t)}{dt} = - A \mathrm{w}(t) + \left( \begin{aligned} -\mathrm{w}_1\mathcal{F}_d(\mathrm{w}_2) \\ \mathrm{w}_1\mathcal{F}_d(\mathrm{w}_2) \end{aligned} \right),
\end{equation}
where
\begin{align*}
	\mathcal{F}_d(\mathrm{w}_2) = \sum_{(x_{ij},y_{ij}) \in \mathcal{Q}(x,y)} w_{i,j}
	g_1(r_i) g_2(\theta_j, x, y) \mathrm{w}_2\bigl(t, x + r_i \cos(\theta_j), y + r_i \sin(\theta_j)\bigr),
\end{align*}
is the quadrature discretization of $\mathcal{F}(\mathrm{w}_2)$.
Note that $\mathcal{F}_d(\mathrm{w}_2)$ has the same form as $T\bigl(t,\mathcal{Q}(x,y)\bigr)$ in
\eqref{T_def}, where $I(t,x,y)$ is replaced by $\mathrm{w}_2(t,x,y)$.
Then, by assumption \eqref{quadr_norm}, equation \eqref{eq:spatialorderproof1} can be rewritten as
\begin{align*}
\dfrac{ d \mathrm{w}(t)}{dt} = - A \mathrm{w}(t) + \left( \begin{aligned} -\mathrm{w}_1\bigl(\mathcal{F}(\mathrm{w}_2) + \bigO{\delta^p}\bigr) \\ \mathrm{w}_1\bigl(\mathcal{F}(\mathrm{w}_2) + \bigO{\delta^p}\bigr) \end{aligned} \right).
\end{align*}
It can be shown that Corollary \ref{col:A1} also holds for $\mathcal{F}_d$.
Therefore, by similar arguments as presented in the proof of Lemma~\ref{contdep}, the equality
$\norm{\tilde{u}-\mathrm{w}}_{l^2} = \bigO{\delta^p}$ holds.
The other estimate, $\norm{\mathrm{w}-\tilde{v}}_{l^2} = \bigO{h^q}$, can be also similarly proved by
rewriting the first two equations of \eqref{trapeq2} (as we did with \eqref{spat_int}) and using the
assumption \eqref{interp_norm}.
\end{proof}

A natural question arises: what is the best type of quadrature and interpolation for solving the system
\eqref{trapeq2}?
In the rest of the section, we describe two numerical integration procedures and also discuss suitable
interpolation techniques.

\subsubsection{Elhay--Kautsky quadrature}
One can use a direct quadrature rule on the general disk, see for example
\cite{Stroud:1971:CalculationMultipleIntegrals, Davis/Rabinowitz:2007:NumericalIntegration}.
In such case the integral of a function $f(x,y)$ over the disk with radius $\delta$ can be approximated by
\begin{align}\label{elhay_kautsky}
	Q(f) = \pi \delta^2 \sum_{i=1}^{N_r \cdot N_\theta} w_i f(x_i, y_i) =
		\pi \delta^2 \sum_{i=1}^{N_r}\sum_{j=1}^{N_\theta} \widetilde{w}_i
		f\bigl(r_i \cos(\theta_j),r_i \sin(\theta_j)\bigr),
\end{align}
where $N_r$ is the number of radial nodes, $N_\theta$ is the number of equally spaced angles, and $w_i$
and $\widetilde{w}_i$ are  weights in the $[0,1]$ interval.
We use $N_\theta = 2 N_r$ to have a quadrature rule that is equally powerful in both $r$ and $\theta$.
The weights and quadrature nodes are calculated by a modification of the Elhay--Kautsky Legendre quadrature
method \cite{Kautsky/Elhay:1982:InterpolatoryQuadratures,
Sylvan/Kautsky:1987:FORTRANInterpolatoryQuadratures, Martin/Wilkinson:1968:ImplicitQL}.
The top panel of Figure~\ref{fig:nodes} shows the distribution of quadrature nodes for
$N_r \in \{3,6,12\}$.
The Elhay--Kautsky quadrature results in nodes that are evenly spaced in the $\theta$ direction.

\subsubsection{Gauss--Legendre quadrature}
Alternatively, we can transform the disk into a square, and then use a one-dimensional Gauss-Legendre
rule to approximate the integral. First, we transform the disk with radius $\delta$ to the rectangle
$[0,\delta] \times [0, 2 \pi]$ in the $r-\theta$ plane. Next, the rectangle
$[0,\delta] \times [0, 2 \pi]$ is mapped to $[0,1] \times [0,1]$ on the $\xi-\eta$ plane by using the
linear transformation
$$r = \delta \xi, \quad \theta = 2 \pi \eta,$$
that has a Jacobian $2 \pi \delta$. Using these transformations, the original integral
\begin{align*}
\int_{0}^{\delta} \int_{0}^{2 \pi} f\bigl(r \cos(\theta), r \sin(\theta)\bigr) \, r \ud\theta \ud r
\end{align*}
takes the form
\begin{align}\label{doubleintegral}
	\int_0^1 \int_0^1 f\bigl(\delta \xi \cos(2 \pi \eta), \delta \xi \sin(2 \pi \eta)\bigr)
		\delta \xi \, 2 \pi \delta \ud\eta \ud\xi.
\end{align}
There are several approaches for computing multiple integrals based on numerical integration of
one-dimensional integrals.
In this paper, we use the Gauss--Legendre quadrature rule on the unit interval
\cite{Trefethen:2008:GaussQuadrature}; other options include generalized Gaussian quadrature rules as
described in \cite{Ma/Rokhlin/Wandzura:1996:GeneralizedQuadrature}.
The integral \eqref{doubleintegral} can be approximated by
\begin{align}\label{gaussQ}
	Q(f) = \sum_{i=1}^{N_\xi} \sum_{j=1}^{N_\eta} w_i w_j 2 \pi \delta^2 \xi_i
	f\bigl(\delta \xi_i \cos(2 \pi \eta_j),\delta \xi_i \sin(2 \pi \eta_j)\bigr) =
	\sum_{m=1}^{N_\xi \cdot N_\eta} \widetilde{w}_{m} f(x_m,y_m),
\end{align}
where $\xi_i$ and $\eta_i$ are the $i$th quadrature nodes corresponding to the Gauss--Legendre quadrature
with weights $w_i$.
The number of quadrature nodes in the $\xi$ and $\eta$ direction are denoted by $N_\xi$ and $N_\eta$,
respectively, and we let $x_m = \delta \xi_i \cos(2 \pi \eta_j)$, $y_m = \delta \xi_i \sin(2 \pi \eta_j)$
and $\widetilde{w}_{m} = w_i w_j 2 \pi \delta^2 \xi_i$.
The distribution of the quadrature nodes in the unit disk is not uniform as with the Elhay--Kautsky
quadrature and can be seen in the bottom panel of Figure~\ref{fig:nodes}.
For a fair comparison we use $N_\eta = 2 N_\xi$.
\begin{figure}[!b]
	\centering
	\subfigure[$3 \times 6$ nodes]
	{\includegraphics[width=0.3\textwidth]{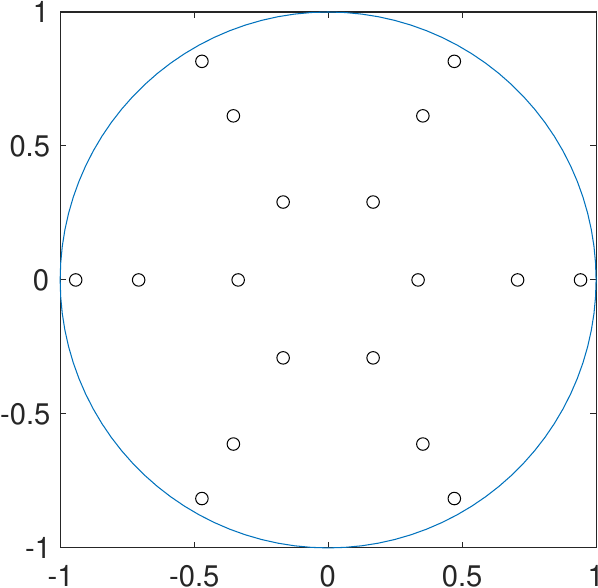}}
	\quad
	\subfigure[$6 \times 12$ nodes]
	{\includegraphics[width=0.3\textwidth]{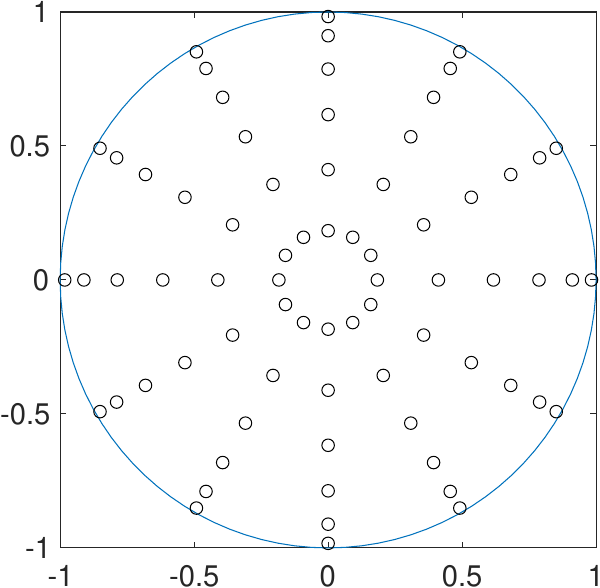}}
	\quad
	\subfigure[$12 \times 24$ nodes]
	{\includegraphics[width=0.3\textwidth]{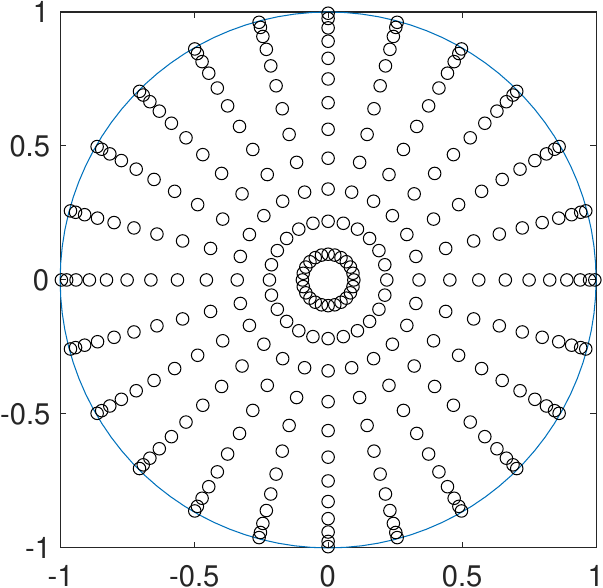}} \\[10pt]
	\subfigure[$3 \times 6$ nodes]
	{\includegraphics[width=0.3\textwidth]{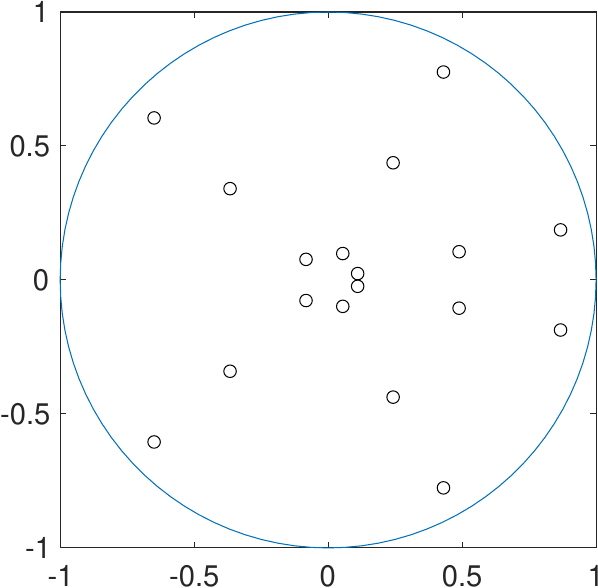}}
	\quad
	\subfigure[$6 \times 12$ nodes]
	{\includegraphics[width=0.3\textwidth]{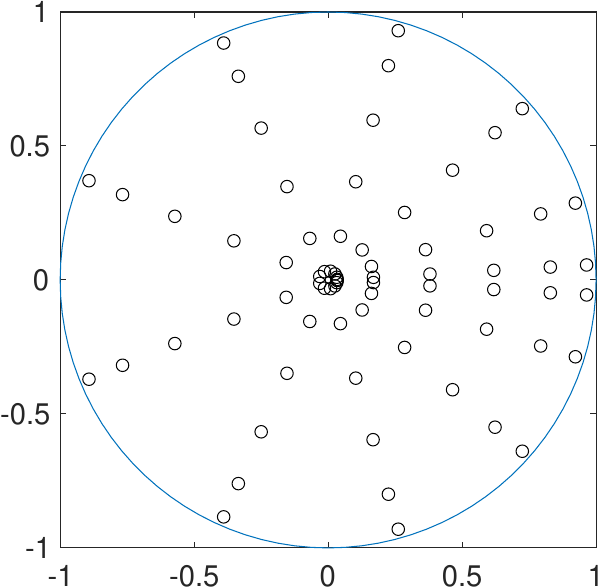}}
	\quad
	\subfigure[$12 \times 24$ nodes]
	{\includegraphics[width=0.3\textwidth]{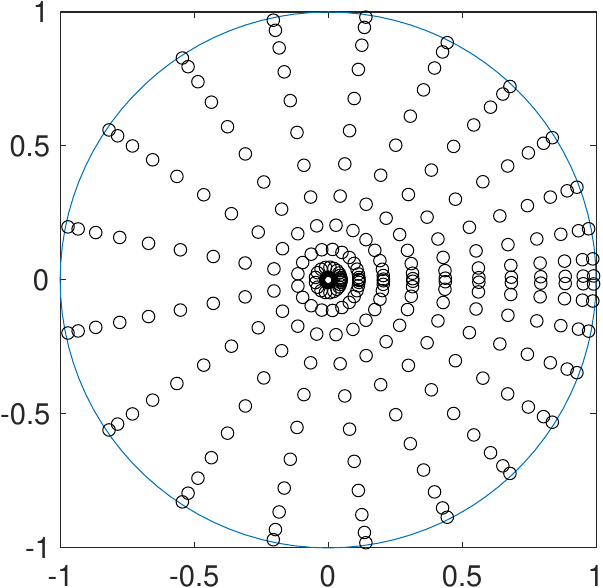}}
	\caption{\textit{Top panel}: The distribution of quadrature nodes ($N_r \times N_\theta$) in the unit
	disk using the Elhay--Kautsky quadrature rule.
	\textit{Bottom panel}: The distribution of quadrature nodes ($N_\xi \times N_\eta$) in the unit disk
	using the Gauss--Legendre quadrature rule.}
    \label{fig:nodes}
\end{figure}
Experimental results reveal that the Elhay--Kautsky quadrature \eqref{elhay_kautsky} performs better in cases
the interpolated function $f(x,y)$ is a bivariate polynomial, whereas the Gauss--Legendre quadrature
\eqref{gaussQ} or the generalized Gaussian quadrature rule (see
\cite{Ma/Rokhlin/Wandzura:1996:GeneralizedQuadrature}) when $f(x,y)$ is an arbitrary nonlinear function.

In order to determine which quadrature rule performs better for the system \eqref{trapeq2}, we perform a
convergence test by applying the quadrature formulas \eqref{elhay_kautsky} and \eqref{gaussQ} to the
function $g_1(r) g_2(\theta,x,y) I_0(r,\theta) r$, where ${g_1(r) = 100(-r + \delta)}$,
$g_2(\theta,x,y) = \sin(\theta) + 1$, and
\begin{align*}
	I_0(r,\theta)  = \frac{100}{2\pi \sigma^2}\exp\left(-\frac{r^2}{2\sigma^2}\right)
\end{align*}
is a Gaussian distribution with deviation $\sigma$ and centered at zero.
This resembles the initial conditions for $I$ at the origin, as we will use later in \cref{sec:Numerics}.
The exact solution of the integral over a disk of radius $\delta$ is given by
\begin{align}\label{ICintegral}
	\int_{0}^{\delta} \int_{0}^{2 \pi} g_1(r) g_2(\theta) I_0 \, r \ud\theta \ud r =
	5000 \left(2 \delta-\sqrt{2\pi}\,\sigma\,\text{erf}\left(\frac{\delta}{\sqrt{2}\sigma}\right)\right),
\end{align}
where $\text{erf}(x)$ is the Gauss error function \cite{Andrews:1998:SpecialFunctions,
William:2002:EconometricAnalysis}.
Figure~\ref{fig:quad_order} shows the convergence of the two quadrature rules over the disk of radius
$\delta$, as $\delta$ goes to zero ($\sigma = 1/10$).
We observe that the Gauss--Legendre quadrature \eqref{gaussQ} gives much smaller errors (close to machine
precision) when more than $12 \times 24$ nodes are used, compared to the Elhay--Kautsky quadrature
\eqref{elhay_kautsky} which is third-order accurate.
\begin{figure}[!t]
\centering
	\subfigure[Elhay--Kautsky quadrature \eqref{elhay_kautsky}]
	{\includegraphics[width=0.47\textwidth]{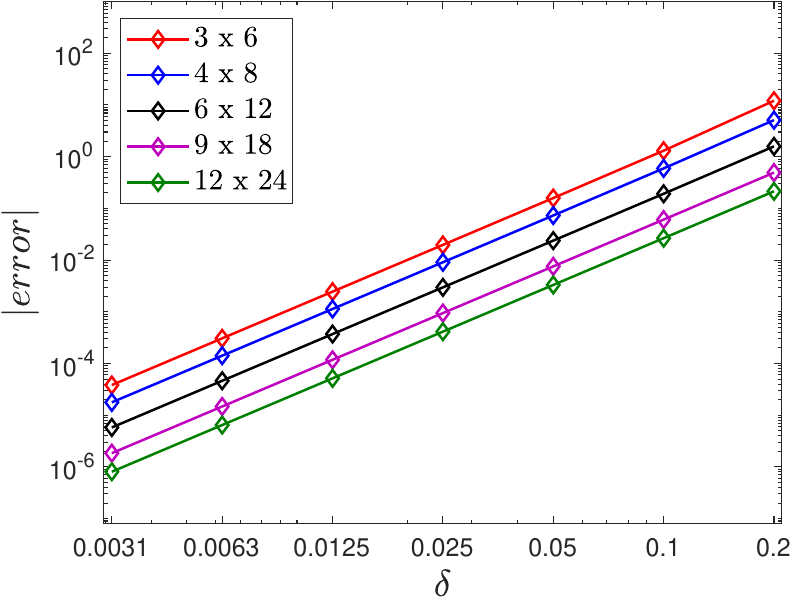}}
	\qquad
	\subfigure[Gauss--Legendre quadrature \eqref{gaussQ}]
	{\includegraphics[width=0.47\textwidth]{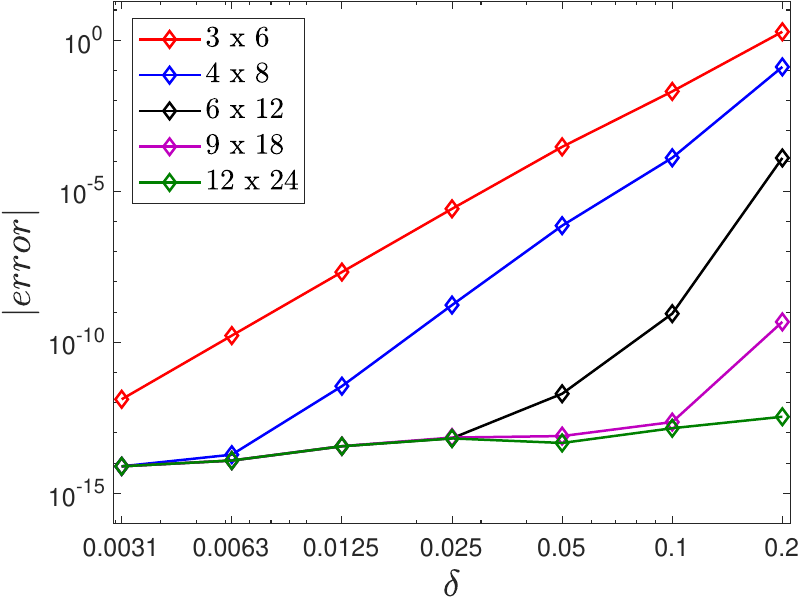}}
	\caption{Numerical integration errors of quadrature formulas \eqref{elhay_kautsky} and \eqref{gaussQ}
	applied to the integral in \eqref{ICintegral}.
	The colored curves correspond to different choices of quadrature nodes in the $\delta$-radius disk.}
    \label{fig:quad_order}
\end{figure}
The performance of the quadrature formulas depends also on the choice and accuracy of interpolation.
As mentioned before, bilinear interpolation can be used since it preserves the non-negativity of the
interpolant.
One possibility is to use higher order interpolations, like cubic or spline, but in these cases the
preservation of the required properties cannot be guaranteed.
However, numerical experiments show that piecewise cubic spline interpolation results in a positive
interpolant for a sufficiently fine spatial grid.
A better choice is the use of a shape-preserving interpolation, to ensure that negative values are not
generated and the interpolant of $I(t, \bar{x}_k, \bar{y}_l)$ in \eqref{quadrature} is bounded by
$\max_{k,l}\{S_{k,l} + I_{k,l} + R_{k,l}\}$ for every point $(x_k, y_l)$.
This can be accomplished by a monotone interpolation that uses \emph{piecewise cubic Hermite interpolating
polynomials}
\cite{Dougherty/Edelman/Hyman:1989:HermiteInterpolation, Fritsch/Carlson:1980:MonotoneInterpolation}.
In \texttt{MATLAB} (version R2021a) the relevant function is called \texttt{pchip} but is only available
for one-dimensional problems.
Extensions to bivariate shape-preserving interpolation have been studied in
\cite{Carlson/Fritsch:1985:MonotoneBicubic, Carlson:Fritsch:1989:MonotoneBicubicAlgorithm,
Fritsch/Carlson:1985:MonotoneBicubicReport}; however, this topic goes beyond the purposes of this
paper.
Another choice is the \emph{modified Akima piecewise cubic Hermite interpolation}, \texttt{makima}.
Numerical experiments demonstrate good performance as it avoids overshoots when more than two
consecutive nodes are constant \cite{Akima:1970:Interpolation, Akima:1974:BivariateInterpolation}, and
hence preserves non-negativity in areas where $I(t, \bar{x}_k, \bar{y}_l)$ is close to zero.
%



\section{Time integration methods}\label{sec:TimeIntegration}
The next step is to use time integration methods to solve the system of ordinary differential equations
\eqref{trapeq2}.
First, we study sufficient and necessary time-step restrictions such that the forward Euler method
satisfies a discrete analogue of properties $C_1$--$C_4$, denoted below by $D_1$--$D_4$.
Then, we discuss how high order SSP Runge--Kutta methods can be applied to \eqref{trapeq2}.

\sloppy{Let $X^n = \lbrace X^n_{k,l} \rbrace$, $X \in \lbrace S,I,R \rbrace$, be the numerical approximation
of $X_{k,l}(t_n)$ for all ${1 \le k \le P_1}$, $1 \le l \le P_2$, and $0 \le n \le \mathcal{N}$, where
$\mathcal{N}$ is the total number of steps.
The numerical solution should satisfy the following properties:
\begin{enumerate}
	\item[$D_1$:] The densities $\lbrace X^n_{k,l} \rbrace$, $X \in \{S, I, R\}$, are non-negative
		for every $ 1 \le k \le P_1$ , ${1 \le l \le P_2}$, and for all $0 \le n \le \mathcal{N}$.
	\item[$D_2$:] The sum $S^n_{k,l} + I^n_{k,l} + R^n_{k,l}$ is constant for all
		$0 \le n \le \mathcal{N}$ and for every $ 1 \le k \le P_1$, $1 \le l \le P_2$.
	\item[$D_3$:] The density $S^n_{k,l}$ is non-increasing, i.e., $S^n_{k,l} \le S^{n-1}_{k,l}$ for every
		$1 \le k \le P_1$, $1 \le l \le P_2$, and for all $1 \le n \le \mathcal{N}$.
	\item[$D_4$:] The density $R^n_{k,l}$ is non-decreasing i.e., $R^n_{k,l} \ge R^{n-1}_{k,l}$ for every
		$1 \le k \le P_1$, $1 \le l \le P_2$, and for all $1 \le n \le \mathcal{N}$.
\end{enumerate}}

\subsection{Explicit Euler scheme and qualitative properties}\label{subsec:FE}
Let us apply the explicit Euler method to the system \eqref{trapeq2} on the interval $[0, \tfinal]$, and
choose an adaptive time step $\tau_n > 0$ such that $t_n = t_{n-1} + \tau_n$, $n \ge 1$.
After the full discretization we get the set of algebraic equalities
\begin{subequations}\label{feuler}
	\begin{empheq}[left ={\empheqlbrace}]{align}
		S^n &= S^{n-1} - \tau_n S^{n-1} \circ T^{n-1} -  c  \tau_n S^{n-1}, \label{feulera} \\
		I^n &= I^{n-1} + \tau_n S^{n-1} \circ T^{n-1} -  b \tau_n I^{n-1}, \label{feulerb} \\
		R^n &= R^{n-1} + b \tau_n I^{n-1} + c  \tau_n S^{n-1}. \label{feulerc}
	\end{empheq}
\end{subequations}
Here, the operator $\circ$ denotes the element-by-element or Hadamard product of matrices.
The matrix $T^{n-1}$ is an approximation of \eqref{quadrature} at all points $(x_k, y_l) \in \mathcal{G}$,
and its components can be expressed by
\begin{align}\label{T_kl}
	T_{k,l}^{n-1} = \sum_{(\bar{x}_k, \bar{y}_l) \in \mathcal{Q}_\delta(x_k,y_l)}
		w_{i,j} g_1(r_i) g_2(\theta_j,x_k, y_l) \tilde{I}^{n-1}(\bar{x}_k, \bar{y}_l).
\end{align}

Now we examine the bounds of time step $\tau_n$ such that the method \eqref{feuler} gives solutions which
are qualitatively adequate and satisfy conditions $D_1$--$D_4$.
\begin{theorem}
Consider the numerical solution \eqref{feuler} obtained by the forward Euler method applied to
\eqref{trapeq2} with non-negative initial data.
Then, the solution satisfies property $D_2$ without any step-size restrictions.
Moreover, properties $D_1$, $D_3$ and $D_4$ hold if the time step satisfies
\begin{align}\label{condtheor3}
	\tau_n \le \min \left\lbrace \dfrac{1}{\max_{k,l}\{T_{k,l}^{n-1}\} + c},
		\dfrac{1}{b} \right\rbrace.
\end{align}
\end{theorem}
\begin{proof}
The proof is similar to the one of \cite[Theorem~2]{Takacs/Horvath/Farago:2019:SpaceModelsDiseases}.
We prove the statement by induction on the number of steps.

First, assume that the properties $D_1$--$D_4$ hold up to step $n-1$; we will prove that they also hold
true for step $n$.
Property $D_2$ can be easily verified by adding all equations in \eqref{feuler}.
To show the monotonicity and non-negativity of $S^n$, consider \eqref{feulera} at point
$(x_k, y_l) \in \mathcal{G}$
\begin{align*}
S_{k,l}^n = \bigl(1 - \tau_n(T_{k,l}^{n-1} + c)\bigr) S_{k,l}^{n-1}.
\end{align*}
By our assumption $I_{k,l}^{n-1} \ge 0$, and a positivity-preserving interpolation guarantees that the
interpolated values
$\tilde{I}^{n-1}(\bar{x}_k,\bar{y}_l) =
\tilde{I}^{n-1}\bigl(x_k + r_i\cos(\theta_j), y_l + r_i\sin(\theta_j)\bigr)$
are non-negative.
Therefore, by \eqref{T_kl} we get $T_{k,l}^{n-1} \ge 0$ for each
$1 \leq k \leq P_1$ , $1 \leq l \leq P_2$ since the weights $w_{i,j}$ are positive, and functions
$g_1$ and $g_2$ are non-negative.
As a result, $\tau_n(T_{k,l}^{n-1} + c) \ge 0$ and thus $S_{k,l}^n \le S_{k,l}^{n-1}$.
Moreover, if $\tau_n \le 1/(T_{k,l}^{n-1} + c)$ then $S_{k,l}^n$ remains non-negative.
Equation \eqref{feulerb} yields
\begin{align*}
I_{k,l}^n = (1-b\tau_n)I_{k,l}^{n-1} + \tau_n S_{k,l}^{n-1}  T_{k,l}^{n-1},
\end{align*}
and hence $I^n$ is non-negative if $\tau_n \le 1/b$.
Finally from \eqref{feulerc} we have
\begin{align*}
R_{k,l}^n = R_{k,l}^{n-1} + b\tau_n I_{k,l}^{n-1} + c \tau_n S_{k,l}^{n-1},
\end{align*}
therefore $R^n$ is non-negative and $R^n \ge R^{n-1}$.
Putting all together we conclude that properties $D_1$--$D_4$ are satisfied if the time step is bounded by
\eqref{condtheor3}.
By using the above argument it can be shown that $D_1$--$D_4$ also hold at the first step, $n = 1$, if the
initial data are non-negative and the time step satisfies \eqref{condtheor3}.
\end{proof}

A drawback of the time-step restriction \eqref{condtheor3} is that it depends on the solution at the
previous step.
This has important complications for higher order methods as we will see in \cref{subsec:RK}.
For any multistage method, the adaptive time step bound \eqref{condtheor3} depends not only on
the previous solution but also on the internal stage approximations.
Consequently, an adaptive time-step restriction based on \eqref{condtheor3} cannot be the same for all
stages of a Runge--Kutta method; instead it needs to be recalculated at every stage to guarantee that
conditions $D_1$--$D_4$ hold.
Therefore, such bound has no practical use because it is prone to rejected steps and will likely tend to
zero.

A remedy is to use a constant time step that is less strict than \eqref{condtheor3}, but still guarantee
${\tau \leq 1/(T_{k,l}^{n-1} + c)}$ holds for all $ 1\leq k \leq P_1$, $1 \leq l \leq P_2$ and at
every step $n$.
At a given point $(x_k,y_l) \in \mathcal{G}$ the weights and quadrature nodes in
$B_{\delta}(x_k,y_l)$ are the same regardless of the location of $(x_k,y_l)$ in the domain.
Therefore, we can find an upper bound for each element of the matrix $T^{n-1}$ in \eqref{T_kl}.
Let
\begin{align}\label{hatT}
	\widehat{T} \coloneqq \sum_{(\bar{x}_k,\bar{y}_l) \in \mathcal{Q}_\delta(x_k,y_l)}w_{i,j}
		g_1(r_i) \kappa_2 M_0 ,
\end{align}
where
\begin{align}\label{tildem}
	M_0 = \max_{(x_k,y_l) \in \mathcal{G}} \left\lbrace S(0,x_k,y_l) + I(0,x_k,y_l) + R(0,x_k,y_l)
		\right\rbrace,
\end{align}
and $\kappa_2$ was defined before. Since $T_{k,l}^{n-1} \le \widehat{T}$ for all $ 1\leq k \leq P_1$,
$1 \leq l \leq P_2$ then if
\begin{align}\label{improvedstepsize}
	\widehat{\tau} \coloneqq \min \left\lbrace \dfrac{1}{\widehat{T}+ c}, \dfrac{1}{b} \right\rbrace,
\end{align}
the condition
\begin{align*}
	\widehat{\tau} \le \min \left\lbrace \dfrac{1}{\max_{k,l}\{T_{k,l}^{n-1}\} + c},
		\dfrac{1}{b} \right\rbrace
\end{align*}
holds at every step $n$.
Moreover, $\widehat{T} \le  \widetilde{w}\, \kappa^2 M_0 N$, where
\begin{align*}
	\kappa = \max\{\kappa_1, \kappa_2\} = \max\left\lbrace\max_{r \in (0, \delta)} \{g_1(r)\},
		\max_{\substack{\theta \in [0, 2 \pi) \\ (x,y) \in \Omega}} \{g_2(\theta,x,y)\}\right\rbrace,
\end{align*}
$\widetilde{w} = \max_{i,j}\{w_{i,j}\}$, and $N$ is the number of the quadrature nodes in
$\mathcal{Q}_\delta(x_k,y_l)$.
Hence, the time step \eqref{improvedstepsize} is larger than the rather pessimistic time step
\begin{align}\label{pessimisticstepsize}
	\widetilde{\tau} \coloneqq \min\left\lbrace\dfrac{1}{\widetilde{w}\,\kappa^2M_0 N+c},
	\dfrac{1}{b}\right\rbrace,
\end{align}
proposed in \cite[Theorem~2]{Takacs/Horvath/Farago:2019:SpaceModelsDiseases}.
Numerical experiments show that $\widehat{\tau}$ is very close to the theoretical bound in
\eqref{condtheor3}, and thus a relatively small increase of time step beyond the bound
\eqref{improvedstepsize}
may produce qualitatively bad solutions which violate one of the conditions $D_1$--$D_4$ (see
\cref{subsec:Eulercomp}).

\subsection{SSP Runge--Kutta methods}\label{subsec:RK}
The forward Euler method is only first-order accurate; hence, we would like to obtain time-step restrictions
for higher order Runge--Kutta methods.
Note that the spatial discretizations discussed in \cref{sec:Spatial} can be chosen so that errors
from quadrature formulas and interpolation are very small; therefore, it is substantial to have a high-order
accurate time integration method.

Consider a Runge--Kutta method in the Butcher form \cite{Butcher:2016:ODEbook} with coefficients
\sloppy{${(a_{ij}) \in \R^{m \times m}}$} and $\bm{b} \in \R^m$.
Let $\mathcal{K}$ be the matrix given by
\begin{align*}
	\mathcal{K} = \left[\begin{array}{cc}
			(a_{ij}) & 0 \\
			\bm{b}^\intercal & 0
		  \end{array}\right],
\end{align*}
and denote by $I$ the $(m+1)$-dimensional identity matrix.
If there exists $r > 0$ such that $(I + r \mathcal{K})$ is invertible, then the Runge--Kutta method can be
expressed in the canonical Shu--Osher form
\begin{align}\label{Shu-OsherRK}
	\begin{split}
	Q^{(i)} &= v_iQ^{n-1}+\sum_{j=1}^m\alpha_{ij}\left(Q^{(j)}+\frac{\tau}{r}F\left(Q^{(j)}\right)\right),
    \qquad 1 \le i \le m+1, \\
    Q^n &= Q^{(m+1)},
	\end{split}
\end{align}
where the coefficient arrays $(\alpha_{ij})$ and $(v_{i})$ have non-negative components.
Such methods are called \emph{strong-stability preserving} (SSP) Runge--Kutta methods and have been
introduced by Shu as total-variation diminishing (TVD) discretizations \cite{Shu:1988:TVD}, and by Shu and
Osher in relation to high order spatial discretizations \cite{Shu/Osher:1988:ENO,Shu:1998:ENO-WENO}.
The choice of parameter $r$ gives rise to different Shu--Osher representations; thus we denote the
Shu--Osher coefficients of \eqref{Shu-OsherRK} by $\balpha_r = (\alpha_{ij})$ and $\bv_r = (v_{i})$ to
emphasize the dependence on the parameter $r$.
The Shu--Osher representation with the largest value of $r$ such that $(I + r \mathcal{K})^{-1}$ exists
and $\balpha_r$, $v_r$ have non-negative components is called optimal and attains the SSP coefficient
\begin{align*}
	\sspcoef = \max\left\lbrace r \ge 0 \;|\; \exists \; (I+r\mathcal{K})^{-1} \text{ and }
		\balpha_r \ge 0, \bv_r \ge 0 \right\rbrace.
\end{align*}
The interested reader may consult \cite{Gottlieb/Ketcheson/Shu:2009:HighOrderSSP,
Gottlieb/Shu:1998:TVDRK, Gottlieb/Shu/Tadmor:2001:SSP}, as well as the monograph
\cite{Gottlieb/Ketcheson/Shu:2011:SSPbook} and the references within, for a throughout review of SSP
methods.

We would like to investigate time-step restrictions such that the numerical solution obtained by applying
method \eqref{Shu-OsherRK} to the problem \eqref{trapeq2} satisfies properties $D_1$--$D_4$.
The following theorem provides the theoretical upper bound for the time step such that these properties
are satisfied.
\begin{theorem}
Consider the numerical solution obtained by applying an explicit Runge--Kutta method \eqref{Shu-OsherRK}
with SSP coefficient $\sspcoef > 0$ to the semi-discrete problem \eqref{trapeq2} with non-negative initial
data.
Then property $D_2$ holds without any time-step restrictions.
Moreover, the properties $D_1$, $D_3$ and $D_4$ hold if the time step satisfies
\begin{align}\label{RK_bound}
	\tau \leq \sspcoef \min \left\lbrace \dfrac{1}{\widehat{T} + c}, \dfrac{1}{b} \right\rbrace,
\end{align}
where $\widehat{T} $ is given by \eqref{hatT}.
\end{theorem}
\begin{proof}
	Consider an arbitrary stage $i$, $1 \le i \le m+1$, of a Runge--Kutta method \eqref{Shu-OsherRK} with
	non-negative coefficients and SSP coefficient $\sspcoef > 0$.
	Applying the method to \eqref{trapeq2} we get
	\begin{subequations}
		\begin{align}
			S^{(i)} &= v_i S^{n-1} + \sum_{j=1}^{i-1} \alpha_{ij}\left(S^{(j)} -
			\frac{\tau}{\sspcoef}\left(S^{(j)} \circ T^{(j)} -c S^{(j)}\right)\right),\label{ith_stageS}\\
			I^{(i)} &= v_i I^{n-1} + \sum_{j=1}^{i-1} \alpha_{ij}\left(I^{(j)} +
			\frac{\tau}{\sspcoef}\left(S^{(j)} \circ T^{(j)} -b I^{(j)}\right)\right),\label{ith_stageI}\\
			R^{(i)} &= v_i R^{n-1} + \sum_{j=1}^{i-1} \alpha_{ij}\left(R^{(j)} +
			\frac{\tau}{\sspcoef}\left(b I^{(j)} + c S^{(j)}\right)\right). \label{ith_stageR}
		\end{align}
	\end{subequations}
	Since all Runge--Kutta methods preserve linear invariants the property $D_2$, i.e.,
	\begin{align*}
		S^n + I^n + R^n = S^{n-1} + I^{n-1} + R^{n-1}, \quad \forall n
	\end{align*}
	is trivially satisfied.

	The remainder of the proof deals with properties $D_1$, $D_3$ and $D_4$.
	We show that all quantities $S^n, I^n, R^n$ remain non-negative, while $S^n$ is non-increasing and
	$R^n$ is increasing.
	From \eqref{ith_stageS} and \eqref{ith_stageI} we have, respectively,
	\begin{align*}
		S^{(i)} &= v_i S^{n-1} + \sum_{j=1}^{i-1} \alpha_{ij}S^{(j)} \circ \left(\mathbf{1} -
			\frac{\tau}{\sspcoef}\left(T^{(j)} + c \mathbf{1} \right)\right), \\
		I^{(i)} &= v_i I^{n-1} + \frac{\tau}{r}\sum_{j=1}^{i-1} \alpha_{ij}S^{(j)} \circ T^{(j)} +
			\left(1 - \frac{\tau}{r}b\right)\sum_{j=1}^{i-1} \alpha_{ij} I^{(j)},
	\end{align*}
	where $\mathbf{1}$ is the $P_1 \times P_2$ all-ones matrix.

	By definition,
	\begin{align*}
		T_{k,l}^{(i)} = \sum_{(\bar{x}_k,\bar{y}_l) \in \mathcal{Q}_\delta(x_k,y_l)}
			w_{i,j} g_1(r_i) g_2(\theta_j, x_k, y_l) \tilde{I}^{(i)}(\bar{x}_k, \bar{y}_l),
			\qquad 1 \le i \le m+1,
	\end{align*}
	where $\tilde{I}^{(i)}$ are interpolated values.
	Since the initial data are non-negative and the chosen interpolation is positivity-preserving, we have
	that $S^{(1)} = S^{n-1}$, $I^{(1)} = I^{n-1}$ and $T^{(1)}$ are all non-negative.
	If
	\begin{align}\label{boundsA}
		 0 \le 1 - \dfrac{\tau}{r}b, \text{ and }
		 0 \le \mathbf{1} - \frac{\tau}{\sspcoef}\left(T^{(j)}+c \mathbf{1} \right) \text{ for }
		 1 \le j \le i-1,
	\end{align}
	then the explicit Runge--Kutta method inductively results in non-negative $T^{(i)}$, $S^{(i)}$, and
	$I^{(i)}$ for each $2 \le i \le m+1$.
	Since $\tilde{I}^{(i)}(\bar{x}_k, \bar{y}_l)$ is an interpolated value of $I^{(i)}(\bar{x}_k, \bar{y}_l)$
	it is bounded by $M_0$ (see \eqref{tildem}).
	Then by using \eqref{hatT}, it holds that $T^{(i)}_{k,l} \le \widehat{T}$, for $1 \le i \le m+1$ and for
	every $(x_k,y_l) \in \mathcal{G}$.
	Therefore,
	\begin{align}\label{boundsB}
		T^{(i)} \le \widehat{T}\mathbf{1}, \quad 1 \le i \le m+1.
	\end{align}
	Moreover, the non-negativity of $T^{(i)}$ implies that
	\begin{align*}
		\mathbf{1} - \frac{\tau}{\sspcoef}\left(T^{(i)}+c \mathbf{1}\right) \le 1, \quad 1 \le i \le m+1,
	\end{align*}
	and thus \eqref{ith_stageS} yields $S^{(i)} \le v_i S^{n-1} + \sum_{j=1}^{i-s} \alpha_{ij}S^{(j)}$.
	Consistency requires that $v_i + \sum_{j=1}^{i-1} \alpha_{ij} = 1$ for each $1 \le i \le m+1$ and
	hence
	\begin{align}\label{S(i)}
		\begin{split}
			S^{(i)} &\le (1-\sum_{j=1}^{i-1} \alpha_{ij})S^{n-1} + \sum_{j=1}^{i-1} \alpha_{ij}S^{(j)} \\
					 &\le S^n - \sum_{j=1}^{i-1} \alpha_{ij}\left(S^{n-1} - S^{(j)}\right).
		\end{split}
	\end{align}	
	Let $1 \le q \le m+1$ be the stage index such that $S^{(i)} \le S^{(q)}$ for all $1 \le i \le m+1$.
	Then, taking $i = q$ in \eqref{S(i)} yields
	\begin{align*}
		S^{(q)} &\le v_i S^{n-1} + \sum_{j=1}^{i-1} \alpha_{qj}S^{(q)} \\
		\left(1 - \sum_{j=1}^{i-1} \alpha_{qj}\right)S^{(q)} &\le
			\left(1 - \sum_{j=1}^{i-1} \alpha_{qj}\right)S^{n-1} \\
		S^{(q)} &\le S^{n-1}.
	\end{align*}
	Therefore, $S^{(i)} \le S^{n-1}$ for all $1 \le i \le m+1$.
	In particular for $i = m+1$ we have ${S^n = S^{(m+1)} \le S^{n-1}}$.
	
	Finally, the non-negativity of initial data, $S^{(j)}$ and $I^{(j)}$ implies that from
	\eqref{ith_stageR} we have $R^{(i)} \ge R^{n-1}$ for all $1 \le i \le m+1$, and hence
	$R^n = R^{(m+1)} \ge R^{n-1}$.
	
	Combining \eqref{boundsA} and \eqref{boundsB} we conclude that the step-size restriction
	\eqref{RK_bound} is sufficient for satisfying properties $D_1$--$D_4$.
\end{proof}


\section{Numerical experiments}\label{sec:Numerics}
In this section, we confirm the results proved previously by using several numerical experiments.
Computational tests are defined in a bounded domain and thus the choice of boundary conditions is
important.
Because we have no diffusion in our problem, we consider homogeneous Dirichlet conditions and we assume
that there is no susceptible population outside of the domain.
This means that we are going to assign a zero value to any point which lies outside of the rectangular
domain in which the problem is defined.
In most cases, the nodes of the quadrature rules \eqref{elhay_kautsky} and \eqref{gaussQ} do not belong
to the spatial grid.
Special attention must be given to the corners and boundaries of the domain where quadrature nodes,
assigned to grid points near the boundary, lie outside of the domain.
In order to handle solution estimates at corners and the boundary of the domain, we use ghost cells
which are set to zero
Thus, we can calculate the values corresponding to the quadrature nodes lying outside of the domain without
violating the qualitative properties.
All code to generate the figures and tables discussed in this section is available at
\url{https://github.com/hadjimy/spatial-SIR_RR}.

For the numerical experiments we are choosing the following functions.
Let $g_1(r)$ be a linearly decreasing function, which takes its maximum at $r=0$ and becomes zero at
$r=\delta$, i.e.,
\begin{align*}
	g_1(r)\coloneqq a(-r + \delta),
\end{align*}
where $a$ is the same parameter as in \eqref{sir}.
The $g_2(\theta,x_k, y_l)$ function is given by
\begin{align}\label{eq:g2}
	g_2(\theta,x_k, y_l) \coloneqq \beta_{k,l} \left(\sin\left(\frac{\pi}{2} + \theta - \alpha_{k,l}\right) +
		\beta_0\right),
\end{align}
where $\alpha_{k,l}$ describes the wind's direction at point $(x_k,y_l)$ and $\beta_{k,l}$ is the strength of
the wind.
The parameter $\beta_0$ is set to $11/10$ to ensure that $g_2(\theta, x_k, y_l)$ is strictly positive.
We use a differentiable velocity field to resemble a wind profile on the domain
$\Omega = [0,\mathcal{L}_1] \times [0,\mathcal{L}_2]$, and hence at each grid point $(k,l)$ the wind
direction is given by a vector $\bm{\upsilon_{k,l}} = (\upsilon^1_{k,l},\upsilon^2_{k,l})$, (see
Figure~\ref{sir_plot}(c)).
The parameter $\alpha_{k,l}$ denotes the angle of the wind vector $\bm{\upsilon_{k,l}}$ with the positive
$x$-axis, and $\beta_{k,l}$ is calculated by the $L^2$-norm of $\bm{\upsilon_{k,l}}$.

\begin{figure}[t!]
	\centering
	\ifarxiv\vspace{-28pt}\fi
	\subfigure[$t = 50$]
	{\includegraphics[width=0.6\textwidth]{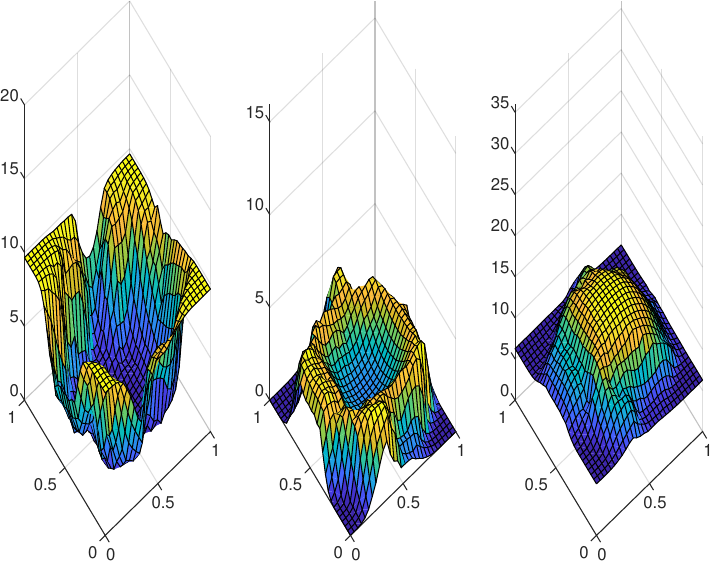}}
	\subfigure[$t = 1000$]
	{\includegraphics[width=0.48\textwidth]{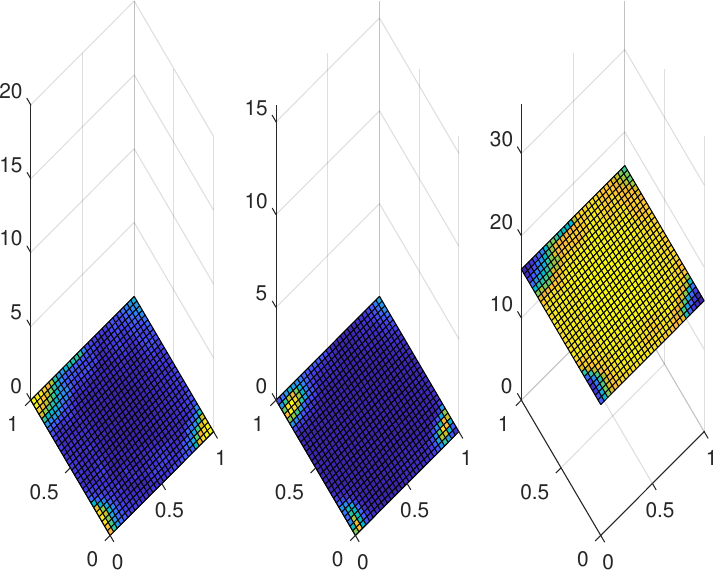}}
	\quad
	\subfigure[wind profile; the colors indicate the intensity of the wind]
	{\includegraphics[width=0.48\textwidth]{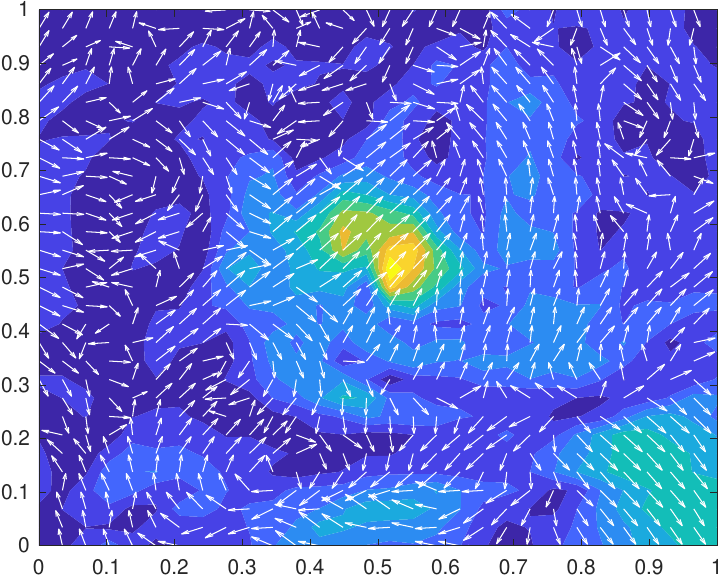}}
	\caption{The number of susceptibles $S$ (left), infected $I$ (middle) and recovered $R$ (right) at times
	$t = 50$ (top panel) and $t = 1000$ (bottom panel).
	The Gauss--Legendre quadrature \eqref{gaussQ} has been used combined with the makima interpolation.
	Plot (c) shows the wind velocity field used.}
	\label{sir_plot}
\end{figure}

The initial conditions resemble the eruption of a wildfire, i.e., having infected cases located in a small
area.
For the infected species, we use a Gaussian distribution concentrated at the middle point
$\left(\mathcal{L}_1/2,\mathcal{L}_2/2\right)$ of the domain $\Omega$, with standard deviation
$\sigma = \min\{\mathcal{L}_1,\mathcal{L}_2\}/10$.
The spatial step sizes are $h_1 = \mathcal{L}_1/(P_1-1)$ and $h_2 = \mathcal{L}_2/(P_2-1)$, where $P_1$
and $P_2$ are the number of grid points in each direction.
We assume that the number of susceptibles is constant except at the middle of the domain, and there are no
recovered species at the beginning.
Therefore, for every $1 \le k \le P_1$, $1 \le l \le P_2$ the initial conditions are given by
\begin{align*}
	I_{k,l}^0 &=  \dfrac{1}{2 \pi \sigma^2} \exp\left( -\dfrac{1}{2} \left[ \left( \dfrac{h_1(k-1) -
	\dfrac{\mathcal{L}_1}{2}}{\sigma}\right)^2 + \left(\dfrac{h_2(l-1)-\dfrac{\mathcal{L}_2}{2}}{\sigma}
	\right)^2\right]\right), \\
	S_{k,l}^0 &= \dfrac{1}{2 \pi \sigma^2}  - I_{k,l}^0, \\
	R_{k,l}^0 &= 0.
\end{align*}
In all numerical experiments - unless otherwise stated - we use the parameter values $a = 100$, $b = 0.05$,
$c = 0.01$, and $\delta = 0.05$.
The computational domain is $\Omega = [0, 1 \times [0, 1]$ with $30$ grid points in each direction, and we
use $6 \times 12$ quadrature nodes.
We also choose the tenth-stage, fourth-order SSP Runge--Kutta method (SSPRK104) for the time integration.

First we would like to study the behavior of our numerical solution. Figure~\ref{sir_plot} depicts the
numerical solution at times $t = 50$ and $t = 1000$.
As we can see, the number of susceptibles is decreased, and the number of infected moves towards the
boundaries, while forming a wave.
Both densities $S$ and $I$ tend to zero, which confirms that the zero solution is indeed an asymptotically
stable equilibrium for the first two equations of \eqref{spat}.

\subsection{Comparison of the step size bounds for the Euler method}\label{subsec:Eulercomp}
As seen in \cref{subsec:FE}, the improved bound $\widehat{\tau}$ (see \eqref{improvedstepsize}) is larger
than the pessimistic bound $\widetilde{\tau}$ (see \eqref{pessimisticstepsize}), and thus closer to the best
theoretically bound \eqref{condtheor3} that guarantees the preservation of properties $D_1$--$D_4$.
We would like to determine how close the bound $\widehat{\tau}$ is to the adaptive step-size
restriction, and compare it with the pessimistic bound $\widetilde{\tau}$.
In Table~\ref{bounds_table} we have tested several different values of $a$ and $\delta$, for which both the
bounds $\widehat{\tau}$ and $\widetilde{\tau}$ were computed.
For comparison we calculated the minimum of the adaptive step bound \eqref{condtheor3}, denoted by
$\tau_e$.
As we can see, varying the parameter $a$ or $\delta$ the time-step bound $\widehat{\tau}$ results in about
$50\%$ increase in efficiency compared to $\widetilde{\tau}$.
Also, the time-step restriction $\widehat{\tau}$ is much closer to the theoretical bound for which the
properties $D_1$--$D_4$ hold.
From Table~\ref{bounds_table} we conclude that in the case of a small increase in the time step
$\widehat{\tau}$, the forward Euler method continues to preserve the desired properties.
However, for values of $\tau$ bigger than \eqref{RK_bound}, there is no guarantee that properties
$D_1$--$D_4$ will be satisfied by a high-order time integration method.

\begin{table}
	\centering
    \small
    \begin{tabular}{c |*{2}{c@{\hskip 15pt}c |} c}
		\toprule
		\multicolumn{1}{c}{$a$} & \multicolumn{1}{c}{$\widetilde{\tau}$} &
		\multicolumn{1}{c}{$\widetilde{\tau}/\tau_e$} & \multicolumn{1}{c}{$\widehat{\tau}$} &
		\multicolumn{1}{c}{$\widehat{\tau}/\tau_e$} & \multicolumn{1}{c}{$\tau_e$} \\
		\midrule
			$50$ 	 & $2.7934$ & $0.4267$ & $6.0609$ & $0.9259$ & $6.5462$ \\
			$100$ 	 & $1.4165$ & $0.4186$ & $3.1252$ & $0.9235$ & $3.3839$ \\
			$250$ 	 & $0.5715$ & $0.4136$ & $1.2740$ & $0.9221$ & $1.3816$ \\
			$500$ 	 & $0.2865$ & $0.4193$ & $0.6411$ & $0.9381$ & $0.6833$ \\
		\bottomrule
	\end{tabular}

	\bigskip

	\begin{tabular}{c |*{2}{c@{\hskip 15pt}c |} c}
		\toprule
		\multicolumn{1}{c}{$\delta$} & \multicolumn{1}{c}{$\widetilde{\tau}$} &
		\multicolumn{1}{c}{$\widetilde{\tau}/\tau_e$} & \multicolumn{1}{c}{$\widehat{\tau}$} &
		\multicolumn{1}{c}{$\widehat{\tau}/\tau_e$} & \multicolumn{1}{c}{$\tau_e$} \\
		\midrule
			$0.025$ & $10.310$ & $0.5155$ & $20.000$ & $1.0000$ & $20.000$ \\
			$0.050$ & $1.4165$ & $0.4186$ & $3.1252$ & $0.9235$ & $3.3839$ \\
			$0.075$ & $0.4239$ & $0.4096$ & $0.9468$ & $0.9149$ & $1.0349$ \\
			$0.100$ & $0.1793$ & $0.4168$ & $0.4016$ & $0.9337$ & $0.4301$ \\
		\bottomrule
	\end{tabular}
	\smallskip
	\caption{Step-size bounds $\widehat{\tau}$ and $\widetilde{\tau}$ (see \eqref{improvedstepsize} and
		\eqref{pessimisticstepsize} respectively), and their comparison with the adaptive bound $\tau_e$
		(see \eqref{condtheor3}) for the forward Euler method for different values of $a$ and $\delta$.
		The computation uses the Elhay--Kautsky quadrature rule \eqref{gaussQ} combined with bilinear
		interpolation, and the final time is $t_\text{f} = 100$.}
    \label{bounds_table}
\end{table}

\subsection{Convergence}\label{subsec:Convergence}
Since we cannot approximate the exact solution accurately, we are going to compute the numerical errors for
different methods by using a reference solution.
To have a fair comparison the reference solution is computed by using the same parameters and method, but
with either a large number of quadrature nodes or a very small time step.

We first observe how well the different quadratures behave. As seen in \cref{sec:Spatial}, using more nodes
in quadrature \eqref{gaussQ} results in smaller errors, and also faster convergence.
Numerical experiments show that this is also the case for the system \eqref{trapeq2}.
The $L^2$-norm errors for the different quadrature formulas and interpolations are plotted in
Figure~\ref{quad_error}.
It is clear that for a small number of quadrature nodes there is no remarkable difference between the
interpolations, but for more quadrature nodes makima and spline interpolation perform better.
Bilinear interpolation results in similar errors for both quadratures \eqref{elhay_kautsky} and
\eqref{gaussQ}.
The makima and spline interpolations have a similar performance for the Elhay--Kautsky quadrature
\eqref{elhay_kautsky} and smaller errors are observed with spline interpolation and Gauss--Legendre
quadrature \eqref{gaussQ}.
Notice, thought that spline interpolation does not guarantee the preservation of properties $D_1$--$D_4$,
e.g., setting $\beta_0 = 1$ in \eqref{eq:g2} yields negative values for the infected density $I$.
\begin{figure}[!t]
\centering
\includegraphics[width=0.55\textwidth]{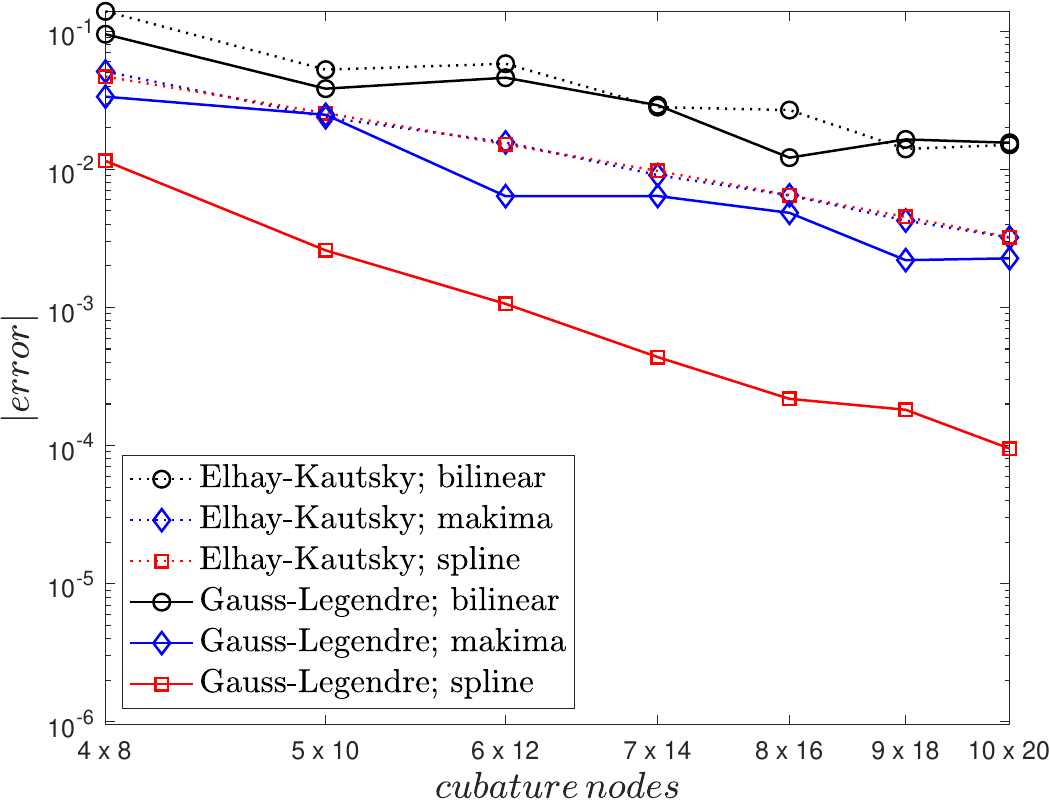}
\caption{$L^2$-norm errors using quadrature formulas \eqref{elhay_kautsky} and \eqref{gaussQ} with
$n \times 2n$ quadrature nodes, $n \in \{4,5,6,7,8,9,10\}$ and different interpolations.
The final time is $t_\text{f} = 50$ and the reference solution for each quadrature rule and interpolation is
computed by using $20 \times 40$ quadrature nodes.}
\label{quad_error}
\end{figure}

Equally important is the order of the different time integration methods.
Table~\ref{first_order_error} shows that the forward Euler method behaves similarly when compared to the
first-order integral solution described in \cref{subsec:InteqralSolution}.
Numerical experiments show that the higher order schemes work as expected, namely that by using enough
quadrature nodes and grid points, a reasonably small error can be achieved with the desired accuracy order.
Table~\ref{higher_order_error} shows the convergence rates for second-, third- and fourth-order SSP
Runge--Kutta methods when the Gauss--Legendre quadrature rule \eqref{gaussQ} is used with spline
interpolation.
The numerical solution is computed at time $\tfinal = 50$ using $30$ grid points and $6 \times 12$ quadrature
nodes.
We start with a reasonable time step $4.7$, which is slightly below the minimum of the adaptive bound
\eqref{condtheor3} when forward Euler method is used, and then successively divide by $2$.
For the reference solution we use a time step that is the half of the smallest time step in our
computations.
It is evident that using higher order methods is better than solving the integral equation \eqref{anasol}
numerically.
Moreover the fourth-order SSP Runge--Kutta method (SSPRK104) attends a six times larger
time step than lower order methods since it has an SSP coefficient $\sspcoef = 6$.
\begin{table}
\centering
    \small
    \begin{tabular}{c *{2}{| c@{\hskip 15pt}c}}
		\toprule
		\multicolumn{1}{c}{$\tau$} & \multicolumn{2}{c}{FE} & \multicolumn{2}{c}{IM} \\
		\midrule	
		$0.8250$ & $4.35 \times 10^{-1}$ &        & $1.74 \times 10^{0}$  & \\
		$0.4125$ & $2.26 \times 10^{-1}$ & $0.94$ & $1.19 \times 10^{0}$  & $0.54$ \\
		$0.2062$ & $1.14 \times 10^{-1}$ & $0.99$ & $7.32 \times 10^{-1}$ & $0.70$ \\
		$0.1031$ & $5.59 \times 10^{-2}$ & $1.03$ & $4.07 \times 10^{-1}$ & $0.85$ \\
		$0.0516$ & $2.62 \times 10^{-2}$ & $1.09$ & $2.06 \times 10^{-1}$ & $0.98$ \\
		$0.0258$ & $1.13 \times 10^{-2}$ & $1.22$ & $9.22 \times 10^{-2}$ & $1.16$ \\
		\bottomrule
	\end{tabular}
	\smallskip
	\caption{$L^2$-norm errors and convergence rates of forward Euler method (FE) and the method
	\eqref{inteq_num}, denoted by ``IM''.
	 The solution is computed at time $t_\text{f} = 50$ with the Gauss--Legendre quadrature rule
	 \eqref{gaussQ}
	 combined with spline interpolation.}
    \label{first_order_error}
\end{table}
\begin{table}
\centering
    \small
    \begin{tabular}{c *{3}{| c@{\hskip 15pt}c}}
		\toprule
		\multicolumn{1}{c}{$\tau$} & \multicolumn{2}{c}{SSPRK22} & \multicolumn{2}{c}{SSPRK33} &
		\multicolumn{2}{c}{SSPRK104} \\
		\midrule
$3.3000$ & $2.93 \times 10^{-1}$ &        & $4.45 \times 10^{-2}$ &        & $4.71 \times 10^{-4}$ & \\
$1.6500$ & $9.08 \times 10^{-2}$ & $1.69$ & $7.13 \times 10^{-3}$ & $2.64$ & $3.30 \times 10^{-5}$ & $3.84$\\
$0.8250$ & $2.53 \times 10^{-2}$ & $1.84$ & $1.01 \times 10^{-3}$ & $2.82$ & $2.18 \times 10^{-6}$ & $3.92$\\
$0.4125$ & $6.63 \times 10^{-3}$ & $1.93$ & $1.35 \times 10^{-4}$ & $2.91$ & $1.40 \times 10^{-7}$ & $3.96$\\
$0.2062$ & $1.63 \times 10^{-3}$ & $2.03$ & $1.72 \times 10^{-5}$ & $2.97$ & $8.82 \times 10^{-9}$ & $3.99$\\
$0.1031$ & $3.29 \times 10^{-4}$ & $2.30$ & $1.94 \times 10^{-6}$ & $3.15$ & $5.19 \times 10^{-10}$& $4.09$\\
		\bottomrule
	\end{tabular}
	\smallskip
	\caption{$L^2$-norm errors and convergence rates of high-order integration methods.
	 The solution is computed at time $t_\text{f} = 50$ with the Gauss--Legendre quadrature rule
	 \eqref{gaussQ}
	 combined with spline interpolation.}
    \label{higher_order_error}
\end{table}

\subsection{Runtime comparison}
The spatial discretization (interpolation and quadrature rule) is expected to be the dominant computation
load for the numerical solution of \eqref{trapeq2}.
We estimate the elapsed time of the numerical solvers for the spatial discretizations; this includes the
calculation of the improved step-size bound $\widehat{\tau}$ (see \eqref{improvedstepsize}) at the
setup of the simulation and all calculations per step for the interpolation and quadrature formulas.
Table~\ref{runtimes} compares the time required (in seconds) for the spatial discretization and the overall
computation time.
The parameters, initial conditions, computational domain, and the number of grid points and quadrature nodes
are the same as at the beginning of this section.
We use three Runge--Kutta methods (forward Euler, SSPRK33, and SSPRK104) combined with the bilinear and
makima interpolation and compute the solution at a final time  $t_\text{f} = 50$.
We choose the Gauss--Legendre quadrature \eqref{gaussQ} with for all tests as it is slightly faster than the
Elhay--Kautsky quadrature.
As shown in Table~\ref{runtimes} the time needed for the spatial discretization is almost equal to the
overall time of the simulations.
Also, the bilinear interpolation is twice as fast as the makima interpolation.
The forward Euler method uses a single function evaluation per step; hence it is the fastest among all
methods.
There is a linear increase in computation time with the three-stage SSPRK33 method because it has the same
SSP coefficient as the forward Euler method but requires three function evaluations per step.
The SSPRK104 method uses ten function evaluations per step; however, it has six times larger step-size bound
than forward Euler and thus it takes only about twice as much time.
The additional computation effort of the SSPRK104 method is compensated by a fourth-order accurate solution.
\begin{table}
\centering
    \small
    \begin{tabular}{c c *{1}{| c@{\hskip 15pt}c}}
		\toprule
		\multirow{2}{*}{method} & \multirow{2}{*}{interpolation} & \multicolumn{2}{c}{time (s)} \\
		& & spatial discretization & total runtime \\
		\midrule
		\multirow{2}{*}{FE      } & bilinear 	& $ 1.417$ & $ 1.453$ \\
					  & makima 	& $ 3.599$ & $ 3.607$ \\[5pt]
		\multirow{2}{*}{SSPRK33 } & bilinear 	& $ 3.964$ & $ 3.974$ \\
					  & makima 	& $10.596$ & $10.606$ \\[5pt]
		\multirow{2}{*}{SSPRK104} & bilinear 	& $ 2.579$ & $ 2.593$ \\
					  & makima 	& $ 6.871$ & $ 6.878$ \\
		\bottomrule
	\end{tabular}
	\smallskip
	\caption{Elapsed time of the spatial discretization compared to the overall simulation time.
	The computations use different Runge--Kutta methods and interpolations, combined with the Gauss--Legendre
	quadrature \eqref{gaussQ} with $30$ grid points in each direction and $6 \times 12$ quadrature nodes.
	The final time is $t_\text{f} = 50$.}
    \label{runtimes}
\end{table}

\section{Conclusions and further work}
In this paper, the SIR model for epidemic propagation is extended to include spatial dependence.
The existence and uniqueness of the continuous solution are proved, along with properties corresponding
to biological observations.
For the numerical solution, different choices of quadrature, interpolation, and time integration methods are
studied.
It is shown that for a sufficiently small time-step restriction, the numerical solution preserves a discrete
analog of the properties of the original continuous system.
The step-size bound is improved compared to previous results.
An adaptive step-size technique is also suggested for the explicit Euler method, and we have determined
step-size bounds for higher order methods.  
Analytic results are confirmed by numerical experiments, while the errors of quadrature formulas and the
order of accuracy of the time discretization methods are also discussed.

The work presented in this paper can be extended to diffusion spatial-dependent SIR systems, and also include
the effect of fractional diffusion.
Results for the preservation of qualitative properties of such a system could be potentially obtained in a
similar fashion as in the current manuscript.
Moreover, the inclusion the births and natural deaths in the system and dropping the conservation property
could make the model more realistic.
Several biological and epidemiological metrics, for instance, the basic reproduction number, could be also
estimated. 
It would be interesting to study the influence of such modification in the behavior of the continuous and
also the numerical solution.

\section*{Acknowledgments}
The research reported in this paper was partially carried out at the Budapest University of Technology and
Economics (BME) and has been supported by the NRDI fund under the auspices of the Ministry for Innovation
and Technology.
The authors would like to thank Lajos L\'oczi for his overall support and suggestions, and Inmaculada
Higueras and David Ketcheson for their comments.

\appendix
\section{Proofs of lemmata in \texorpdfstring{\cref{sec:Analytical}}{}}\label{appx:ProofLemmas}
We present the proofs of some technical lemmata that were omitted in the previous
sections.

\begin{proof}[Proof of Lemma~\ref{lem:normequivalence}]
The statement simply follows from
\begin{align*}
	\norm{u}^2 = c^2 \dfrac{1}{c^2} \norm{u_1}_{L^2}^2 + b^2 \dfrac{1}{b^2} \norm{u_2}_{L^2}^2 &\leq
	\max\left\lbrace\dfrac{1}{b^2},\dfrac{1}{c^2}\right\rbrace
	\left(c^2 \norm{u_1}_{L^2}^2  + b^2 \norm{u_2}_{L^2}^2\right) \\
	& = \left(\frac{1}{\min\{b,c\}}\right)^2 \norm{u}_A^2,
\end{align*}
and
\begin{align*}
	\norm{u}_A^2 = c^2 \norm{u_1}_{L^2}^2 + b^2 \norm{u_2}_{L^2}^2 \leq
	\max\{b^2,c^2\} (\norm{u_1}_{L^2}^2  + \norm{u_2}_{L^2}^2 )
	= \max\{b,c\}^2 \norm{u}^2.
\end{align*}
\end{proof}
\begin{proof}[Proof of Lemma~\ref{lem:L2boundedness}]
We are going to derive an upper bound to the term
\begin{align*}
	\norm{\mathcal{F}(u_2)}_{L^2}^2 &= \bigintsss_{\Omega} \bigg\vert \int_{0}^{\delta}
		\int_{0}^{2 \pi} g_1(r) g_2(\theta,x,y) u_2\bigl(t, \bar{x}(r,\theta), \bar{y}(r,\theta)\bigr) \, r
		\ud\theta \ud r \bigg\vert^2 \! \ud x \ud y \\
	&= \bigintsss_{\Omega} \bigg\vert\int_{B_{\delta}(\mathbf{x})} g_1(r)g_2(\theta,x,y)
		u_2(t, \bar{\mathbf{x}}) \ud\bar{\mathbf{x}} \bigg\vert^2 \! \ud\mathbf{x},
\end{align*}
where we used the notation
$\bar{\mathbf{x}} \coloneqq \bigl(\bar{x}(r,\theta),\bar{y}(r,\theta)\bigr) =
\bigl(x + r\cos(\theta), y + r\sin(\theta)\bigr)$, and $B_{\delta}(\mathbf{x})$ is the ball with radius
$\delta$ around $\mathbf{x}$.
By the definition of $g_1$ and $g_2$, we have that
\begin{align*}
	\norm{\mathcal{F}(u_2)}_{L^2}^2 = \bigintsss_{\Omega} \bigg\vert\int_{\Omega}
		g_1(r) g_2(\theta,\mathbf{x}) u_2(t, \tilde{\mathbf{x}}) \ud\tilde{\mathbf{x}} \bigg\vert^2
		\ud\mathbf{x}.
\end{align*}
We also know that $g_1$ and $g_2$ are bounded.
Using the notations $\kappa_1 = \max_{r \in (0, \delta)} \{g_1(r) \}$ and
$\kappa_2 = \max_{\theta \in [0, 2 \pi), \mathbf{x} \in \Omega} \{g_2(\theta, \mathbf{x}) \}$, yields
\begin{align*}
	\norm{\mathcal{F}(u_2)}_{L^2}^2 &=
		\bigintsss_{\Omega}\bigg\vert\int_{\Omega} 1\cdot g_1(r) g_2(\theta,\mathbf{x})
		u_2(t, \tilde{\mathbf{x}}) \ud\tilde{\mathbf{x}} \bigg\vert^2 \! \ud\mathbf{x} \\
		&\le \!\!\! \bigintss_{\Omega}\Biggg\vert \sqrt{\int_{\Omega} 1^2 \ud\tilde{\mathbf{x}}} \;
			\sqrt{\int_{\Omega} \big( g_1(r) g_2(\theta,\mathbf{x}) u_2(t, \tilde{\mathbf{x}})\big)^2
			\ud\tilde{\mathbf{x}}}\Biggg\vert^2 \! \ud\mathbf{x} \\
		&\le \mu(\Omega) \int_{\Omega} \int_{\Omega}
		\bigl\vert g_1(r) g_2(\theta,\mathbf{x}) u_2(t, \tilde{\mathbf{x}}) \bigr\vert^2
			\ud \tilde{\mathbf{x}} \ud\mathbf{x} \\
		&\le \kappa_1^2 \, \kappa_2^2 \, \mu(\Omega) \int_{\Omega} \int_{\Omega}
			\bigl\vert u_2(t, \tilde{\mathbf{x}}) \bigr\vert^2 \ud\tilde{\mathbf{x}} \ud\mathbf{x},
\end{align*}
where we used the Cauchy--Schwarz inequality, and $\mu(\Omega)$ is the Lebesgue measure of
$\Omega$.
It holds that
\begin{align*}
	\int_{\Omega}\int_{\Omega}\bigl\vert u_2(t,\tilde{\mathbf{x}})\bigr\vert^2
		\ud\tilde{\mathbf{x}} \ud\mathbf{x}
	= \int_{\Omega} \norm{u_2}_{L^2}^2 \ud\mathbf{x} = \mu(\Omega) \norm{u_2}_{L^2}^2.
\end{align*}
Consequently,
\begin{align*}
	\norm{\mathcal{F}(u_2)}_{L^2} \leq \kappa_1 \, \kappa_2 \, \mu(\Omega) \norm{u_2}_{L^2},
\end{align*}
and setting $\nu_{\mathcal{F}} = \kappa_1 \kappa_2 \, \mu(\Omega)$ we get the result of the lemma.
\end{proof}

\begin{proof}[Proof of Lemma~\ref{contdep}]
Consider the system \eqref{spat_epsilona}-\eqref{spat_epsilonb} written in the compact form \eqref{geneq},
where $A$ is given by \eqref{operatorA}, $F$ is defined in \eqref{Fdef}, and the corresponding norm
$\norm{\cdot}$ is given by \eqref{norm_E}.
Let $\{\varepsilon_i\}$ and $\{\varepsilon_j\}$ be two sequences such that
$\lim_{i\rightarrow\infty} \varepsilon_i = 0$ and $\lim_{j\rightarrow\infty} \varepsilon_j = 0$.
Assume that $u_{\varepsilon_i}(t)$ and $u_{\varepsilon_j}(t)$ are solutions of
\eqref{spat_epsilona}-\eqref{spat_epsilonb}, and define the vectors
$\widehat{\varepsilon_i} \coloneqq (0, \varepsilon_i)^\intercal$ and
$\widehat{\varepsilon_j} \coloneqq (0, \varepsilon_j)^\intercal$
Then
\begin{align*}
	u'_{\varepsilon_i} (t) - u'_{\varepsilon_j}(t) = A (u_{\varepsilon_i} (t) - u_{\varepsilon_j} (t)) +
	F(u_{\varepsilon_i}(t)) - F(u_{\varepsilon_j}(t)) + \widehat{\varepsilon_i} - \widehat{\varepsilon_j}.
\end{align*}
Using the definition of $A$ and Corollary~\ref{col:A1}, yields
\begin{align*}
	\norm{u'_{\varepsilon_i} (t) - u'_{\varepsilon_j}(t)} &\leq \norm{A (u_{\varepsilon_i} (t) -
		u_{\varepsilon_j} (t))} + \norm{F(u_{\varepsilon_i}(t)) - F(u_{\varepsilon_j}(t))} +
		\norm{\widehat{\varepsilon_i} - \widehat{\varepsilon_j}} \\
	&\leq \bigl(\norm{A} + \zeta(d)\bigr) \norm{u_{\varepsilon_i} (t) - u_{\varepsilon_j} (t)} +
		\norm{\widehat{\varepsilon_i} - \widehat{\varepsilon_j}},
\end{align*}
where $\zeta(d)$ is defined in Corollary~\ref{col:A1}.
By Grönwall's inequality (see \cite[Lemma~1.6]{Halanay:ODEs:1966}), we have
\ifwias
\begin{align*}
	\norm{u_{\varepsilon_i} (t) - u_{\varepsilon_j} (t)} &\leq
		\norm{\widehat{\varepsilon_i} - \widehat{\varepsilon_j}} t \\
		&\quad + \norm{\int_0^t \bigl(\norm{A} + \zeta(d)\bigr) \; (\widehat{\varepsilon_i} -
		\widehat{\varepsilon_j})\; t \exp \left( \int_s^t \bigl(\norm{A} + \zeta(d)\bigr) \; \ud w \right) \ud s}.
\end{align*}
\else
\begin{align*}
	\norm{u_{\varepsilon_i} (t) - u_{\varepsilon_j} (t)} \leq
		\norm{\widehat{\varepsilon_i} - \widehat{\varepsilon_j}} t +
		\norm{\int_0^t \bigl(\norm{A} + \zeta(d)\bigr) \; (\widehat{\varepsilon_i} - \widehat{\varepsilon_j})\; t
		\exp \left( \int_s^t \bigl(\norm{A} + \zeta(d)\bigr) \; \ud w \right) \ud s}.
\end{align*}
\fi
Since we assume that $\lim_{i,j\rightarrow\infty} \varepsilon_i - \varepsilon_j = 0$, the
statement is proved.
\end{proof}

\ifarxiv
	\pagebreak
\fi
\section{List of symbols and notations}
\begin{table}[!ht]
\centering
\def\arraystretch{1.3}
\begin{tabular}{M{0.3\textwidth}|L{0.6\textwidth}}
\toprule
Symbol & \multicolumn{1}{c}{Description} \\
\midrule
$S(t,x,y)$, $I(t,x,y)$, $R(t,x,y)$ & density functions of susceptible, infected and recovered species \\
\hline
$a$, $b$, $c$ & parameters describing the rate of infection, recovery, and \nobreak vaccination \\
\hline
$\delta$ & parameter describing the radius of the effect of an infectious \nobreak individual \\
\hline
$g_1(r)$, $g_2(\theta, x, y)$ & functions describing the effect of an infectious individual to its
surroundings \\
\hline
$E = L^2(\Omega) \times L^2(\Omega)$ & Hilbert-space for the solution of system \eqref{geneq} \\
\hline
$\norm{ \cdot }$ & norm of Hilbert-space $E$ defined in \eqref{norm_E} \\
\hline
$\norm{ \cdot }_A$ & operator norm of $D(A)$ defined in \eqref{norm_DA} \\
\hline
$A : D(A) \rightarrow E$ & linear part of equation \eqref{spat} \\
\hline
$F : D(A) \rightarrow E$ & non-linear part of equation \eqref{spat} \\
\hline
$\mathcal{F}$ & the integral term in \eqref{spat} \\
\hline
\ifnum0%
	\ifarxiv 1 \fi \ifwias 1 \fi > 0
$\phi (t,x,y)$ & time integral of the term $\mathcal{F}$ used in the integral solution \eqref{anasol} \\
\hline
$M_0 (x,y)$ & initial sum of densities at point $(x,y) \in \Omega$ \\
\hline
$\mathcal{Q}_{\delta}(x,y)$ & points of the quadrature used to approximate the term $\mathcal{F}$ defined in
the ball $B_{\delta}(x,y)$ \\
\hline
$w_{i,j}$ & weights of the quadrature $\mathcal{Q}_{\delta} (x,y)$ \\
\hline
$T(t, \mathcal{Q}_{\delta} (x,y))$ & approximation of the term $\mathcal{F}$ using the quadrature
$\mathcal{Q}_{\delta}(x,y)$ defined in \eqref{T_def} \\
\hline
$\Omega = [0,\mathcal{L}_1] \times [0,\mathcal{L}_2]$ & spatial domain of the semi-discretization
problem \eqref{spat_int} \\
\hline
\fi
\ifarxiv
$\mathcal{G}$ & spatial grid for the domain $\Omega$ \\
\hline
$S_{k,l}(t)$, $I_{k,l}(t)$, $R_{k,l}(t)$ & approximations of the continuous solution at point
$(x_k, y_l) \in \mathcal{G}$ \\
\hline
$Q(f)$ & numerical approximation of the integral of function $f$ over the disc with radius $\delta$ \\
\hline
$S_{k,l}^n$, $I_{k,l}^n$, $R_{k,l}^n$ & approximations of $S_{k,l}(t)$, $I_{k,l}(t)$ and $R_{k,l}(t)$ at
time $t_n$ \\
\hline
$S^n$, $I^n$, $R^n$ & matrices containing the elements $S_{k,l}^n$, $I_{k,l}^n$ and $R_{k,l}^n$ \\
\hline
\fi
\end{tabular}
\end{table}

\begin{table}[!ht]
\centering
\def\arraystretch{1.3}
\begin{tabular}{M{0.3\textwidth}|L{0.6\textwidth}}
\ifjournal
\hline
$\phi (t,x,y)$ & time integral of the term $\mathcal{F}$ used in the integral solution \eqref{anasol} \\
\hline
$M_0 (x,y)$ & initial sum of densities at point $(x,y) \in \Omega$ \\
\hline
$\mathcal{Q}_{\delta}(x,y)$ & points of the quadrature used to approximate the term $\mathcal{F}$ defined in
the ball $B_{\delta}(x,y)$ \\
\hline
$w_{i,j}$ & weights of the quadrature $\mathcal{Q}_{\delta} (x,y)$ \\
\hline
$T(t, \mathcal{Q}_{\delta} (x,y))$ & approximation of the term $\mathcal{F}$ using the quadrature
$\mathcal{Q}_{\delta}(x,y)$ defined in \eqref{T_def} \\
\hline
$\Omega = [0,\mathcal{L}_1] \times [0,\mathcal{L}_2]$ & spatial domain of the semi-discretization
problem \eqref{spat_int} \\
\fi
\ifnum0%
	\ifjournal 1 \fi \ifwias 1 \fi > 0
\hline
$\mathcal{G}$ & spatial grid for the domain $\Omega$ \\
\hline
$S_{k,l}(t)$, $I_{k,l}(t)$, $R_{k,l}(t)$ & approximations of the continuous solution at point
$(x_k, y_l) \in \mathcal{G}$ \\
\hline
$Q(f)$ & numerical approximation of the integral of function $f$ over the disc with radius $\delta$ \\
\hline
$S_{k,l}^n$, $I_{k,l}^n$, $R_{k,l}^n$ & approximations of $S_{k,l}(t)$, $I_{k,l}(t)$ and $R_{k,l}(t)$ at time
$t_n$ \\
\hline
$S^n$, $I^n$, $R^n$ & matrices containing the elements $S_{k,l}^n$, $I_{k,l}^n$ and $R_{k,l}^n$ \\
\fi
\hline
$T^n$ & matrix containing the approximations of $T(t, \mathcal{Q}_{\delta} (x,y))$ at time $t_n$ defined in
\eqref{T_kl} \\
\hline
$\tau_n$ & time step of the time integration method \\
\hline
$M_0$ & constant describing the maximum of the initial density of the population \\
\hline
$\hat{\tau}$, $\tilde{\tau}$ & time step bounds for the forward Euler method \\
\hline
$a_{ij}$, $\boldsymbol{b}$ & coefficients of the Runge--Kutta method in the Butcher form \\
\hline
$v_i$, $\alpha_{ij}$, $r$ & coefficients of the Runge--Kutta method in the Shu--Osher form \\
\hline
$\mathcal{C}$ & SSP coefficient of the Runge--Kutta method \\
\hline
$S^{(i)}$, $I^{(i)}$, $R^{(i)}$ & matrices containing the approximations of the solution at internal stages
of the Runge-Kutta method \\
\hline
$T^{(i)}$ & matrix containing the values of the approximations of $T(t, \mathcal{Q}_{\delta} (x,y))$ at
internal stages of the Runge-Kutta method \\
\hline
$\alpha_{k,l}$, $\beta_{k,l}$, $\beta_0$, $v_{k,l}^1$, $v_{k,l}^2$ & constants related to the wind profile \\
\bottomrule
\end{tabular}
\end{table}

\ifarxiv
	\pagebreak
\fi
\ifnum0%
	\ifjournal 1 \fi \ifwias 1 \fi > 0

	\ifjournal
		\bibliographystyle{elsarticle-harv}
	\else
		\bibliographystyle{acm}
	\fi
	\bibliography{references}

\else

	\printbibliography

\fi

\end{document}